\documentclass[11pt]{article}
\usepackage{amsmath,amssymb,amsthm,amsfonts,mathrsfs}
\usepackage{enumitem,graphics,xcolor,yfonts,colonequals}
\usepackage{stmaryrd}
\usepackage[alphabetic]{amsrefs}
\usepackage[all,cmtip,color]{xy}
\usepackage[colorlinks,anchorcolor=blue,citecolor=blue,linkcolor=blue,urlcolor =blue,bookmarksdepth=2,bookmarksopen=true]{hyperref}
\usepackage{comment}
\usepackage{mathtools}
\DeclarePairedDelimiter{\ceil}{\lceil}{\rceil}

\usepackage{fancyhdr}
\fancypagestyle{titlepage}
{
	\fancyhf{}

	\fancyfoot[l]{
	\href{https://mathscinet.ams.org/mathscinet/msc/msc2020.html}
	}
}

\fancypagestyle{writing_sample}{
    \fancyhead{}
	\rhead{Lai, Kuan-Wen - Writing Sample}
}

\AtBeginDocument{%
	\def\MR#1{}
}

\newcommand{\bN}{\mathbb{N}}    
\newcommand{\bZ}{\mathbb{Z}}    
\newcommand{\bQ}{\mathbb{Q}}    
\newcommand{\bC}{\mathbb{C}}    
\newcommand{\bF}{\mathbb{F}}    

\newcommand{\cA}{\mathcal{A}}   
\newcommand{\cC}{\mathcal{C}}   
\newcommand{\bD}{\mathrm{D^b}}  
\newcommand{\cM}{\mathcal{M}}   
\newcommand{\cO}{\mathcal{O}}   
\newcommand{\bP}{\mathbb{P}}    

\newcommand{\disc}{\mathrm{d}}	
\newcommand{\mukaiH}{\widetilde{H}} 
\newcommand{\PGL}{\mathrm{PGL}}	
\newcommand{\Pic}{\mathrm{Pic}}	

\newtheorem*{thm*}{Theorem}
\newtheorem*{prop*}{Proposition}
\newtheorem*{cor*}{Corollary}
\newtheorem{thm}{Theorem}[section]
\newtheorem{prop}[thm]{Proposition}
\newtheorem{cor}[thm]{Corollary}
\newtheorem{lemma}[thm]{Lemma}
\numberwithin{equation}{section}

\theoremstyle{definition}

\newtheorem{rmk}[thm]{Remark}

\begin{document}
\title{New rational cubic fourfolds arising from Cremona transformations}
\author{Yu-Wei Fan
    \and Kuan-Wen Lai}
\date{}

\newcommand{\ContactInfo}{{
\bigskip\footnotesize

\bigskip
\noindent Y.-W.~Fan,
\textsc{Yau Mathematical Sciences Center, Tsinghua University\\
Jinchunyuan West Bldg., Room 302\\
Beijing 100084, China}\par\nopagebreak
\noindent\textsc{Email:} \texttt{yuweifanx@gmail.com}

\bigskip
\noindent K.-W.~Lai,
\textsc{Mathematisches Institut der Universit\"at Bonn\\
Endenicher Allee 60, 53121 Bonn, Deutschland}\par\nopagebreak
\noindent\textsc{Email:} \texttt{kwlai@math.uni-bonn.de}
}}

\maketitle
\thispagestyle{titlepage}

\begin{abstract}
Are Fourier--Mukai equivalent cubic fourfolds birationally equivalent? We obtain an affirmative answer to this question for very general cubic fourfolds of discriminant $20$, where we produce birational maps via the Cremona transformation defined by the Veronese surface. By studying how these maps act on the cubics known to be rational, we surprisingly found new rational examples.
\end{abstract}

\setcounter{tocdepth}{3}
\tableofcontents

\section{Introduction}
\label{sect:intro}

For every smooth complex cubic hypersurface $X\subseteq\bP^5$, that is, a cubic fourfold, its bounded derived category of coherent sheaves $\bD(X)$ contains a full triangulated subcategory called the \emph{K3 category of $X$}:
$$
    \cA_X\colonequals
    \left<\cO_X,\cO_X(1),\cO_X(2)\right>^\perp
    \subseteq\bD(X).
$$
While the K3 category of a very general cubic fourfold is \emph{not} equivalent to $\bD(S)$ for any K3 surfaces $S$, such an equivalence does hold for certain special cubic fourfolds and, in these special cases, some of the cubics are known to be birational to $\bP^4$. This motivates Kuznetsov to conjecture that a cubic fourfold is rational if and only if its K3 category is equivalent to $\bD(S)$ for some K3 surface $S$; see \cite{Kuz10}*{Conjecture~1.1}.

When the K3 category $\cA_X$ is realizable by a polarized K3 surface of degree $14$, Beauville and Donagi proved that a very general such $X$ is rational \cite{BD85}. For the limiting cases, the rationality was proved by Bolognesi--Russo--Staglian\`o \cite{BRS19}, as well as Auel \cite{Aue22}, while the same result can also be derived as a consequence from the later work \cite{KT19}. In the cases that $\cA_X$ is realizable by a polarized K3 surface of degrees $26$, $38$, or $42$, the rationality of $X$ was proved by Russo--Staglian\`o \cites{RS18,RS19,RS23} by appealing to \cite{KT19}.

To what extent does the K3 category determine the birational geometry of a cubic fourfold? For a very general cubic $X$, it is known that a cubic $X'$ is isomorphic to $X$ if and only if their K3 categories are equivalent \cite{Huy17}*{Theorem~1.5~(i)}. For special cubic fourfolds, do equivalences between K3 categories guarantee the existence of \emph{birational maps} between them \cite{MS19}*{Question~3.25}?

In this paper, we focus on the cubic fourfolds containing a Veronese surface $V\subseteq\bP^5$. Let $\cC = [U/\PGL_6(\bC)]$ denote the moduli space of cubic fourfolds, where $U\subseteq|\cO_{\bP^5}(3)|$ is the subset of smooth cubics. Then the cubics containing (a degeneration of) $V$ determine a divisor $\cC_{20}\subseteq\cC$.
On the other hand, the system of quadrics containing $V$ defines a \emph{Cremona transformation} of $\bP^5$, i.e. a birational map
$$\xymatrix{
	F_V\colon\bP^5\ar@{-->}[r]^-\sim & \bP^5.
}$$
which is an involution upon composing with a projective transformation \cite{CK89}*{Theorem~3.3}.
Our first main result is as follows.

\begin{thm}[= Theorem~\ref{thm:cremonaFM}]
\label{thm:birationalFM-Partners}
By taking $X\subseteq\bP^5$ to its proper image $X'\colonequals F_V(X)\subseteq\bP^5$, the map $F_V$ induces a birational involution
\[\xymatrix{
    \sigma_V\colon\cC_{20}\ar@{-->}[r]^-\sim & \cC_{20}
}\]
such that for a very general $X\in\cC_{20}$, the image $X'$ appears as the unique cubic fourfold such that $\cA_X\cong\cA_{X'}$ and $X\ncong X'$.
\end{thm}

Cubic fourfolds with equivalent K3 categories are called \emph{Fourier--Mukai partners}.
Note that, for a very general $X\in\cC_{20}$, the map $F_V$ restricts as a birational map between $X$ and its proper image $X'$, so Theorem~\ref{thm:birationalFM-Partners} implies immediately the following.

\begin{cor}
\label{cor:birationalK3Cat}
For very general cubics $X, X'\in\cC_{20}$, if they are Fourier--Mukai partners, then they are birational to each other.
\end{cor}

How does the map $\sigma_V$ act on the locus in $\cC_{20}$ which parametrizes rational cubic fourfolds?
Before answering this question, let us briefly review the background:
For a very general $X\in\cC$, the algebraic lattice
\[
    A(X)\colonequals
    H^{2,2}(X,\bC)\cap H^4(X,\bZ)
\]
is spanned by $h^2$, the square of the hyperplane class.
A member $X\in\cC$ is called \emph{special} if $A(X)$ contains a rank~$2$ saturated sublattice
\[
    A(X)\supseteq K\ni h^2
\]
called a \emph{labelling}.
According to Hassett \cite{Has00}, special cubic fourfolds admitting a labelling of discriminant $d$ form an irreducible divisor $\cC_d\subseteq\cC$, which is nonempty if and only if
\begin{equation}
\label{eqn:discriminantNonempty}
    d\geq 8
    \quad\text{and}\quad
    d\equiv 0,2\text{ (mod }6).
\end{equation}
Moreover, for every $X\in\cC_d$, there is an equivalence $\cA_X\cong\bD(S)$ for some K3 surface $S$ if and only if $d$ is \emph{admissible}, namely,
\begin{equation}
\label{eqn:discriminantK3}
    d\text{ is not divisible by }
    4, 9,
    \text{ or any odd prime }
    p\equiv 2\text{ (mod }3).
\end{equation}
For this fact, we refer the reader to \cite{BLMNPS21}*{Corollary~1.7} and the references cited there. A list of such $d\leq 200$ is provided in \cite{Add16}*{Table~1}. Notice that $d=20$ satisfies \eqref{eqn:discriminantNonempty} but not \eqref{eqn:discriminantK3}.

A very general cubic is expected to be irrational, though no such example has been found so far. For rational examples, the four types of cubics mentioned in the beginning constitute four divisors $\cC_{14}$, $\cC_{26}$, $\cC_{38}$, and $\cC_{42}$ in $\cC$. Besides these, there are codimension~$1$ loci in $\cC_8$, see \cite{Has99}, and $\cC_{18}$, see \cite{AHTV19}, known to parametrize rational cubics. Singular cubics, which are characterized by their limiting Hodge structures possessing certain labellings of discriminants $2$ and $6$, see \cite{Has00}*{\S4.2 and \S4.4} (see also \cite{Has16}*{\S2.3}), are also rational.

Since the Cremona map $F_V$ can possibly transform a smooth cubic into a singular one, it is necessary to consider the closure $\overline{\cC_{20}}$ of $\cC_{20}$ in the Laza--Looijenga compactification $\overline{\cC}$, see \cite{Laz10}*{Theorem~1.2}, and extend $\sigma_V$ as a birational involution on $\overline{\cC_{20}}$. For simplicity, we still use $\sigma_V$ to denote this extension.
The loci of singular cubics of discriminants $2$ and $6$ in $\overline{\cC}$ will be denoted by $\cC_2$ and $\cC_6$, respectively.

\begin{thm}
\label{thm:newRationalCubics}
For each $d = 26, 38, 42$, the birational involution $\sigma_V$ maps a component of $\cC_{20}\cap\cC_d$ birationally onto a component of $\overline{\cC_{20}}\cap\cC_{d'}$ where $d'$ cannot be in the list
\[
    \{2,\, 6,\, 8,\, 14,\, 18,\, 26,\, 38,\, 42\}.
\]
As a consequence, there exist at least three irreducible divisors in $\cC_{20}$ consisting of rational cubic fourfolds which were not known before.
\end{thm}

More details about this theorem are provided in Theorem~\ref{thm:newRationalCubics_details}. The situation about $\cC_{20}\cap\cC_{14}$ is summarized in Remark~\ref{rmk:C20C14}.

In fact, a cubic in $\cC_{20}\cap\cC_{d}$ possesses infinitely many distinct labellings simultaneously for every admissible $d$.  The following theorem shows that, for each admissible $d$, there exists a cubic in $\sigma_V(\cC_{20}\cap\cC_{d})$ with admissible discriminants among which the minimal one is strictly greater than $d$. In particular, the map $\sigma_V$ can potentially produce new rational cubic fourfolds whenever a new divisor $\cC_d$ is found to parametrize rational cubics.

\begin{thm}[= Theorem~\ref{thm:rationalbiggerdisc}]
\label{thm:rationalbiggerdisc_intro}
Let $d\geq 14$ be an even integer which is admissible.
Then $\sigma_V(\cC_{20}\cap\cC_d)$ contains a component $D$ such that
\begin{enumerate}[label=\textup{(\arabic*)}]
\setlength\itemsep{0pt}
    \item $D\not\subseteq\cC_{d'}$ for any admissible $d'$ with $d'\leq d$,
    \item $D\subseteq\cC_{d'}$ for some admissible $d'$ with $d'>d$.
\end{enumerate}
\end{thm}

This paper is organized as follows: In Section~\ref{sect:veroCubic}, we review the necessary background on cubic fourfolds containing a Veronese surface and establish a few propositions required in proving the main results. These include a birational model for the Veronese locus $\cC_{20}$, and formulas counting the Fourier--Mukai partners of a very general $X\in\cC_d$ with $d$ not divisible by $9$. In Section~\ref{sect:involVeroLocus}, we first introduce the Cremona transformation defined by the Veronese surface, then analyze its induced action on $\cC_{20}$, and finally study the action on the locus of rational cubic fourfolds. Throughout the paper, we say that a member in a moduli space is \emph{very general} provided that it is in the complement of a countably infinite union of divisors.

\subsection*{Acknowledgements}
We thank Brendan Hassett for introducing the fundamental example upon which this work is built. We also thank Asher Auel for proposing the second main question in the introduction. This work also benefitted from the discussions with Emanuele Macr\`{i} and Eyal Markman. Finally, we thank the anonymous referee for their helpful comments.

\section{Cubic fourfolds containing a Veronese surface}
\label{sect:veroCubic}

The main purpose of this section is to establish two results for the use of Section~\ref{sect:involVeroLocus} and briefly review necessary background during the process. The first result is about a birational model for $\cC_{20}$, which ensures that a cubic in it contains a Veronese surface if it does not contain a plane. The second result gives the number of Fourier--Mukai partners of a very general cubic in $\cC_d$ with $d$ not divisible by $9$. As a special case, it will imply that a very general cubic in $\cC_{20}$ has one and only one Fourier--Mukai partner not isomorphic to itself.

\subsection{A birational model for the Veronese locus}
\label{subsect:modelVeroLocus}

Recall that a cubic fourfold $X$ is special if and only if the lattice
$$
    A(X)\colonequals H^{2,2}(X,\bC)\cap H^4(X,\bZ)
$$
contains a labelling. Because the integral Hodge conjecture is proved for cubic fourfolds \cite{Voi07}*{Theorem~18}, this lattice is generated by the classes of algebraic cycles. As a consequence, $X$ is special if and only if it contains an algebraic surface not homologous to a complete intersection.

Special cubic fourfolds containing a Veronese surface determine a divisor $\cC_{20}\subseteq\cC$ which we call the \emph{Veronese locus}. Before making this statement precise, let us briefly review the basic construction of this locus. First of all, one can produce an example of a \emph{smooth} cubic fourfold $X$ containing a Veronese surface in the following way.

Recall that the Veronese surface $V\subseteq\bP^5$ is defined as the embedding of $\bP^2$ into $\bP^5$ via the linear system of conics. Upon taking a projective transformation, we can write this embedding as
$$
    \bP^2\hookrightarrow\bP^5:
    [x:y:z]\mapsto[x^2:xy:y^2:yz:z^2:zx].
$$
If we denote by $[X_0:\dots:X_5]$ the homogeneous coordinates of $\bP^5$, then the $2\times 2$ minors of the matrix
\begin{equation}
\label{eqn:deterVero}
\begin{pmatrix}
X_0 & X_1 & X_5\\
X_1 & X_2 & X_3\\
X_5 & X_3 & X_4
\end{pmatrix}
\end{equation}
form a basis for the ideal $I_V$ of $V\subseteq\bP^5$. Using this explicit description, one can easily produce an example of smooth cubic $X$ containing $V$ with the aid of a computer algebra system (for example, \textsc{Singular} \cite{DGPS}).

Now let $X\subseteq\bP^5$ be a smooth cubic containing the surface $V$. As stated in \cite{Has00}*{\S4.1.4}, the sublattice
$
    K_V\colonequals\langle
        h^2, [V]
    \rangle
    \subseteq A(X)
$
is isometric to
\begin{equation}
\label{eqn:labelling20}
    K_V\cong\begin{pmatrix}
        h^2\cdot h^2 & h^2\cdot[V]\\
        h^2\cdot[V] & [V]\cdot[V]
    \end{pmatrix}
    \cong\begin{pmatrix}
        3 & 4\\
        4 & 12
    \end{pmatrix}
\end{equation}
which has discriminant $20$. Furthermore, $K_V$ is saturated since any nontrivial finite-index overlattice of $K_V$ must have discriminant $5$, which does not satisfy condition \eqref{eqn:discriminantNonempty}, and thus $K_V$ is a labelling of discriminant $20$.

Every automorphism of $\bP^5$ preserving $V$ is extended uniquely from an action of $\PGL_3(\bC)$ on $V\cong\bP^2$. This defines an action of $\PGL_3(\bC)$ on $|I_V(3)|$ and thus on the open subset $U_{20}\subseteq|I_V(3)|$ that parametrizes smooth members. Therefore, we can form the quotient $[U_{20}/\PGL_3(\bC)]$ in the sense of geometric invariant theory \cite{MFK94}. This determines a morphism
\begin{equation}
\label{eqn:quotientToC20}
\xymatrix{
    \varphi\colon [U_{20}/\PGL_3(\bC)]\ar[r] & \cC_{20}.
}
\end{equation}
Every cubic fourfold with a Veronese surface lies in the image of $\varphi$ as all the Veronese surfaces form a single $\PGL_6(\bC)$-orbit. In the following, we prove that a member of $\cC_{20}$ lies in the image of $\varphi$ once we know it is \emph{not} in the divisor $\cC_{8}\subseteq\cC$ that parametrizes cubic fourfolds containing a plane.

\begin{prop}
\label{prop:containsVeronese}
The morphism $\varphi$ is birational and its image contains the open subset $\cC_{20}\setminus\cC_{8}$. In particular, every member of $\cC_{20}\setminus\cC_{8}$ contains a Veronese surface.
\end{prop}

\begin{proof}
First we prove that $\varphi$ is birational. Suppose that there exists a cubic fourfold $X$ containing a $1$-dimensional family of Veronese surfaces. Since all Veronese surfaces are projectively equivalent, this implies that $X$ admits a $1$-dimensional stabilizer in $\PGL_6(\bC)$, which cannot happen \cite{MFK94}*{Proposition~4.2}. Therefore, a cubic fourfold contains at most finitely many Veronese surfaces. This shows that $\varphi$ has finite fibers, and so is dominant as its domain and codomain have the same dimension. For a very general member $X\in\cC_{20}$, the lattice $A(X)$ has rank~$2$, which implies that $X$ contains one and only one Veronese surface. This proves that $\varphi$ is birational.

Now we prove that every $X_0\in\cC_{20}\setminus\cC_8$ has a preimage under $\varphi$. Because $\varphi$ is dominant, there exists a deformation of $X_0$ such that a general fiber is a cubic fourfold containing a Veronese surface. More precisely, there exists a family of cubic fourfolds $\mathcal{X}\to D$ over an open disk $\{0\}\in D\subseteq\bC$ and a family of Veronese surfaces $\mathcal{V}\to D\setminus\{0\}$ which form a commutative diagram
\[\xymatrix{
    \mathcal{V}\ar[d]\ar@{^(->}[r]
    & \mathcal{X}\ar[d]
    & X_0\ar@{_(->}[l]\ar[d]\\
    D\setminus\{0\}\ar@{^(->}[r]
    & D
    & \{0\}.\ar@{_(->}[l]
}\]
Let $V_0\subseteq X_0$ denote the specialization of $\mathcal{V}$ over $0\in D$. Our goal is to show that $V_0$ is a Veronese surface.

First we show that $V_0$ is an integral surface. Since $V_0$ has degree~$4$, if it is not integral, then it either involves a plane as a component or consists of two possibly singular quadric surfaces which may coincide or not. The former case is ruled out as $X_0\notin\cC_8$. (A cubic fourfold belongs to $\cC_8$ if and only if it contains a plane.) In the latter case, $X_0$ contains a quadric $Q$. Then $Q$ spans a $3$-space $P$ in $\bP^5$ which intersects $X_0$ in the union of $Q$ with a plane, but this is impossible as $X\notin\cC_8$. This proves that $V_0$ is integral.

Next, we show that $V_0$ is nondegenerate. Assume, to the contrary, that $V_0$ is contained in a hyperplane $H\subseteq\bP^5$. Define $Y\colonequals H\cap X_0$. Then $V_0\subseteq Y$ is a divisor, and by the Lefschetz hyperplane theorem, every divisor of $Y_0$ has degree divisible by $3$. This is a contradiction as $V_0$ has degree~$4$.

According to \cite{SH73} (see also \cite{Har10}*{Exercise~29.6(c)}), every nondegenerate integral surface of degree~$4$ in $\bP^5$ is one of the following:
\begin{enumerate}[label=(\roman*)]
\setlength\itemsep{0pt}
    \item\label{deg4InP5_P1xP1}
    the embedding of $\bP^1\times\bP^1$ via the linear system $|\cO_{\bP^1}(1)\oplus\cO_{\bP^1}(2)|$,
    \item\label{deg4InP5_Hirzebruch2}
    the embedding of the Hirzebruch surface $\bF_2$ via the linear system $|C_0+3f|$, where $C_0$ is the unique sectional class with $C_0^2 = -2$ and $f$ is the fiber class,
    \item\label{deg4InP5_veronese}
    the Veronese surface,
    \item\label{deg4InP5_cone}
    a cone over a rational quartic curve in $\bP^4$.
\end{enumerate}

Surfaces in cases \ref{deg4InP5_P1xP1} and \ref{deg4InP5_Hirzebruch2} have Euler characteristic $\chi = 4$ while the Veronese surface has $\chi = 3$, so these cases can be ruled out. Suppose that case \ref{deg4InP5_cone} occurs, and let $p\in V_0$ be the cone vertex. The tangent hyperplane $T_pX_0$ of $X_0$ at $p$ contains every line in $X_0$ passing through $p$ and thus contains the rulings of $V_0$. But this implies that $V_0$ is contained in $T_pX_0\cong\bP^4$ and so is degenerate, leading to a contradiction. We conclude that only case \ref{deg4InP5_veronese} can happen.
\end{proof}

\begin{rmk}
Here we address what might happen to a Veronese surface $V$ contained in a smooth cubic fourfold $X$ when $X$ deforms to a very general member $X_0\in\cC_{20}\cap\cC_{8}$. Let $V_0\subseteq X_0$ denote the specialization of $V$. By hypothesis, $V_0 = P\cup S$, where $P$ is a plane and $S$ is an integral surface of degree~$3$. If $S$ spans a $\bP^4$, then it is a \emph{variety of minimal degree}, which implies that it is either a type $(1,2)$ rational normal scroll, or a cone over a twisted cubic \cite{EH87}*{Theorem~1}. If $S$ span a $\bP^3$, then it is a possibly singular cubic surface. In order to keep the contents of this paper concise, we leave the questions about what the intersection $P\cap S$ might be, which configuration truly occurs, and which one induces a Cremona transformation open.
\end{rmk}

\subsection{Basic facts about the transcendental lattices}
\label{subsect:transcendentalLattice}

Given a cubic fourfold $X$, its \emph{transcendental lattice} is defined as the orthogonal complement
\[
	T(X)\colonequals A(X)^\perp\subseteq H^4(X, \bZ).
\]
Note that it carries a Hodge structure inherited from $H^4(X,\bZ)$. The purpose of this section is to recall some basic facts and standard results about $T(X)$ which will be used in Section~\ref{subsect:countFM} and later.
In the following, given a lattice $\Lambda$, we will denote by $\Lambda^*\colonequals\mathrm{Hom}(\Lambda,\bZ)$ its dual lattice and by $\disc\Lambda\colonequals\Lambda^*/\Lambda$ its discriminant group.

\begin{lemma}
\label{lemma:isometryT(X)}
Let $X$ be a very general cubic fourfold in $\cC_{d}$. Then the only Hodge isometries on $T(X)$ are $1$ and $-1$.
\end{lemma}

\begin{proof}
One can use almost the same proof as that of \cite{Ogu02}*{Lemma~(4.1)} to prove this proposition. The only part which requires further verification is the minimality condition: If $T\subseteq T(X)$ is a minimal saturated sublattice such that 
\begin{equation}
\label{eqn:containPeriod}
    H^{3,1}(X,\bC)\subseteq T\otimes\bC
\end{equation}
then $T = T(X)$. Assume, to the contrary, that $T\subsetneq T(X)$ is a minimal saturated sublattice such that \eqref{eqn:containPeriod} holds. Then $T$ has a nonzero orthogonal complement $T^\perp\subseteq T(X)$, and \eqref{eqn:containPeriod} implies that $T^\perp$ is orthogonal to $H^{3,1}(X,\bC)$. It follows that $T^\perp$ is of type $(2,2)$. This implies that $T^\perp$ is algebraic by the integral Hodge conjecture \cite{Voi07}*{Theorem~18}, but this contradicts the fact that $T^\perp\subseteq T(X)$.
Therefore, the minimality condition holds in our case.
\end{proof}

\begin{lemma}
\label{lemma:dTcyclic}
Let $X$ be a very general cubic fourfold in $\cC_{d}$, where $d$ is not divisible $9$.
Then $dT(X)$ is cyclic.
\end{lemma}

\begin{proof}
Since $H^4(X,\bZ)$ is unimodular, we have $dT(X)\cong dA(X)$.
If $d \equiv2\ (\mathrm{mod}\ 6)$, then for a very general cubic $X\in\cC_d$,
\[
A(X) \cong
\begin{pmatrix}
    3 & 1 \\
    1 & \frac{d+1}{3}
\end{pmatrix}.
\]
If $d \equiv0\ (\mathrm{mod}\ 6)$ and $d$ is not divisible by $9$, then for a very general cubic $X\in\cC_d$,
\[
A(X) \cong
\begin{pmatrix}
    3 & 0 \\
    0 & \frac{d}{3}
\end{pmatrix}.
\]
One can check that $dA(X)$ is cyclic in both cases.
Therefore, $dT(X)$ is cyclic.
\end{proof}

The transcendental lattice can also be constructed naturally from the K3 categories $\cA_X$. The main reference for the following is \cites{AT14}. Define
\[
    \mukaiH(\cA_X, \bZ)\colonequals K_\textup{top}(\cA_X)
\]
as the topological Grothendieck group,
which has a lattice structure under the Euler pairing
\[
    \chi(E,F)\colonequals
    \sum_{i\in\bZ}(-1)^i\dim\mathrm{Hom}_D(E, F[i]).
\]
As an abstract lattice, we have
\[
    \mukaiH(\cA_X, \bZ)\cong E_8(-1)^{\oplus2}\oplus U^{\oplus 4}
\]
where the right-hand side coincides with the Mukai lattice of a K3 surface.
The Mukai vector induces an injection \cite{AH61}*{\S2.5}
\[\xymatrix{
    \textbf{v}\colon
    \mukaiH(\cA_X, \bZ)\ar@{^(->}[r] & H^*(X, \bQ):
    E\mapsto\operatorname{ch}(E)\sqrt{\operatorname{td(X)}}
}\]
which defines a weight $2$ Hodge structure on $\mukaiH(\cA_X, \bZ)$ by
\begin{align*}
    \mukaiH^{2,0}(\cA_X, \bC)
    &\colonequals \textbf{v}^{-1}H^{3,1}(X, \bC),\\
    \mukaiH^{1,1}(\cA_X, \bC)
    &\colonequals \textbf{v}^{-1}\left(\oplus_{n=0}^4H^{n,n}(X, \bC)\right),\\
    \mukaiH^{0,2}(\cA_X, \bC)
    &\colonequals \textbf{v}^{-1}H^{1,3}(X, \bC).
\end{align*}
As analogues of the N\'eron--Severi lattice and the transcendental lattice of a K3 surface, let us define
\begin{align*}
    N(\cA_X)
    &\colonequals\mukaiH^{1,1}(\cA_X,\bC)\cap\mukaiH(\cA_X,\bZ)\\
    T(\cA_X)
    &\colonequals N(\cA_X)^\perp\subseteq\mukaiH(\cA_X, \bZ).
\end{align*}

The objects $[\cO_\text{line}(1)]$ and $[\cO_\text{line}(2)]$ in $D^b(X)$ induce a sublattice
\[A_2 =
\begin{pmatrix}
    2 & -1\\
    -1 & 2
\end{pmatrix}
\subseteq N(\cA_X).
\]
There may be multiple $A_2$ sublattices in $\mukaiH(\cA_X,\bZ)$, though all of them can be identified via $O(\mukaiH(\cA_X,\bZ))$. We denote the one coming from $[\cO_\text{line}(1)]$ and $[\cO_\text{line}(2)]$ by $A_2(X)$.
By \cite{AT14}*{Proposition~2.3}, restricting $\textbf{v}$ to the orthogonal complement $A_2(X)^\perp\subseteq\mukaiH(\cA_X,\bZ)$ induces a Hodge isometry
\begin{equation}
\label{eqn:a2InPrim}
\xymatrix{
    A_2(X)^{\perp}\ar[r]^-\sim & H^4(X, \bZ)_\mathrm{prim}(-1).
}
\end{equation}
Further restrictions induce Hodge isometries
\begin{equation}
\label{eqn:mukaiAA}
\xymatrix{
    N(\cA_X)\cap A_2(X)^{\perp}\ar[r]^-\sim
    & \left(A(X)\cap H^4(X, \bZ)_\mathrm{prim}\right)(-1)
}
\end{equation}
\begin{equation}
\label{eqn:mukaiTT}
\xymatrix{
        T(\cA_X)\ar[r]^-\sim & T(X)(-1).
}
\end{equation}

\begin{lemma}[\cite{Has00}*{Proposition 3.2.2}]
\label{lemma:neron-60}
Assume that $X\in\cC_{d}$ is very general. Then
\[
    N(\cA_X)\cap A_2(X)^{\perp}
    \cong
    \left(
    A(X)\cap H^4(X, \bZ)_\mathrm{prim}
    \right)(-1)
\]
is a rank $1$ lattice $\left<\ell\right>$.
Moreover,
\[
\ell^2 = 
\begin{cases}
-3d, & \text{if } d \equiv2\ (\mathrm{mod}\ 6), \\
-\frac{d}{3}, & \text{if } d \equiv0\ (\mathrm{mod}\ 6).
\end{cases}
\]
\end{lemma}

\subsection{Counting the Fourier--Mukai partners}
\label{subsect:countFM}

This section aims to compute the number of Fourier--Mukai partners for a very general $X\in\cC_{d}$ when $d$ is not divisible by $9$. Following \cite{Huy17}, we say that two cubic fourfolds $X$ and $Y$ are \emph{Fourier--Mukai partners} if there exists an equivalence $\cA_X\xrightarrow{\sim}\cA_Y$ which is of Fourier--Mukai type, i.e.~such that the composition
\[\xymatrix{
    D^b(X)\ar[r] & \cA_X\ar[r]^-\sim & \cA_Y \ar@{^(->}[r] & D^b(Y)
}\]
is a Fourier--Mukai transform.
If $X\in\cC_d$ is very general, then this is equivalent to the existence of a Hodge isometry \cite{Huy17}*{Theorem~1.5~(iii)}
\[\xymatrix{
    F\colon\mukaiH(\cA_X,\bZ)\ar[r]^-\sim & \mukaiH(\cA_Y,\bZ).
}\]
By restricting to the transcendental parts, we obtain a commutative diagram
\begin{equation}
\label{eqn:mukaiToTrans}
\begin{aligned}
\xymatrix{
    \mukaiH(\cA_X,\bZ)\ar[r]^{F}_-{\sim}
    & \mukaiH(\cA_Y,\bZ)\\
    T(\cA_X)\ar@{^(->}[u]\ar[r]^-{\sim}
    & T(\cA_Y).\ar@{^(->}[u].
}
\end{aligned}
\end{equation}

Now suppose that $X$ is a very general member of $\cC_{d}$ and $d$ is not divisible by $9$. Then $Y$ also belongs to $\cC_{d}$ and is very general since $T(X)\cong T(Y)$ by \eqref{eqn:mukaiTT}.
Note that the isometry on the bottom of \eqref{eqn:mukaiToTrans} extends uniquely to an isometry on the top by \cite{Nik79}*{Theorem~1.14.4} and the fact that $\disc T(X)$ is cyclic (Lemma~\ref{lemma:dTcyclic}). Therefore, $X$ and $Y$ are Fourier--Mukai partners if and only if $T(\cA_X)$ and $T(\cA_Y)$ are isomorphic as Hodge lattices.

The number of Fourier--Mukai partners of a very general cubic fourfold $X\in\cC_d$ for admissible $d$ has been computed by Pertusi \cite{Per21}*{Theorem~1.1}. In order to treat the case $d=20$, we generalize it to the following.

\begin{prop}
\label{prop:FMPartner}
Let $X$ be a very general cubic fourfold in $\cC_{d}$, where $d$ satisfies (\ref{eqn:discriminantNonempty}) and is not divisible by $9$.
Define a number $m\in\bN$ depending on $d$:
\begin{itemize}
    \item $m=1$ if $d=2^a$;
    \item $m=2^{k-1}$ if $d=2p_1^{e_1}\cdots p_k^{e_k}$;
    \item $m=2^{k}$ if $d=2^ap_1^{e_1}\cdots p_k^{e_k}$.
\end{itemize}
Here $a\geq2$, and the $p_i$ are distinct odd primes.
Then:
\begin{enumerate}[label=\textup{(\arabic*)}]
    \item If $d \equiv2\ (\mathrm{mod}\ 6)$, then the number of Fourier--Mukai partners of $X$ equals $m$.
    \item If $d \equiv0\ (\mathrm{mod}\ 6)$ and not divisible by $9$, then the number of Fourier--Mukai partners of $X$ equals $\frac{1}{2}m$.
\end{enumerate}
In particular, if $X$ is a very general cubic in $\cC_{20}$, then the number of Fourier--Mukai partners of $X$ equals $2$.
\end{prop}

\begin{proof}
We follow closely the idea in \cite{Ogu02} which counts the numbers of Fourier--Mukai partners of K3 surfaces with Picard rank $1$.
Fix a very general $X\in\cC_{d}$. We define
\[
    T\colonequals T(\cA_X)
    \quad\text{and}\quad
    S\colonequals N(\cA_X)\cap A_2(X)^\perp\cong \langle \ell\rangle
\]
where $\ell^2 = -3d$ when $d \equiv2\ (\mathrm{mod}\ 6)$ and $\ell^2 = -\frac{1}{3}d$ in the other case. We further define $\cM_{S,T}$ to be the collection of even overlattices $L\supseteq S\oplus T$ which satisfy
\begin{itemize}
\setlength\itemsep{0pt}
    \item $S\oplus T\subseteq L\subseteq S^*\oplus T^*$;
    \item $S$ and $T$ are both saturated in $L$;
    \item $L$ has discriminant $3$; that is, $[L^*:L]=3$.
\end{itemize}
Note that each $L\in\cM_{S,T}$ is equipped with the weight $2$ Hodge structure induced from $T$. (For the definition of \emph{even overlattices}, see \cite{Nik79}*{\S1.4}.)

Let $\mathrm{FM}(X)$ denote the set of Fourier--Mukai partners of $X$. Our goal is to prove that
\[
    |\mathrm{FM}(X)|
    = \begin{cases}
    m, & \text{if } d \equiv2\ (\mathrm{mod}\ 6),\\
    \frac{m}{2}, & \text{if } d \equiv0\ (\mathrm{mod}\ 6) \text{ and }9\nmid d.
    \end{cases}
\]
We accomplish this via the relation between $\mathrm{FM}(X)$ and $\cM_{S,T}$ as described below:
Let $Y\in\cC_{d}$ be very general. By Lemma~\ref{lemma:neron-60}, there exist exactly two choices of isometries
\begin{equation}
\label{eqn:SS_Y}
\xymatrix{
    \phi\colon S\ar[r]^-\sim &
    N(\cA_Y)\cap A_2(Y)^\perp =\vcentcolon S_Y
}
\end{equation}
such that one is the negative of the other. Assume further that $Y\in\mathrm{FM}(X)$. Then there exists an isometry
\begin{equation}
\label{eqn:TT_Y}
\xymatrix{
    \psi\colon T\ar[r]^-\sim &
    T(\cA_Y) =\vcentcolon T_Y
}
\end{equation}
respecting the Hodge structures. By Lemma~\ref{lemma:isometryT(X)}, there are exactly two such isometries, where one is the negative of the other.
These induce an isometry on the dual spaces
\[\xymatrix{
    (\phi^*\oplus \psi^*): S_Y^*\oplus T_Y^* \ar[r]^-\sim &
    S^*\oplus T^*.
}\]
Note that $A_2(Y)^\perp\in\cM_{S_Y,T_Y}$.
Define
\begin{equation}
\label{eqn:overlattice}
    L_{Y,\phi,\psi}\colonequals (\phi^*\oplus \psi^*)(A_2(Y)^\perp).
\end{equation}
Then $ L_{Y,\phi,\psi}\in\cM_{S,T}$. Also note that the restriction
\[\xymatrix{
    (\phi^*\oplus \psi^*)|_{A_2(Y)^\perp}\colon
    A_2(Y)^\perp\ar[r]^-\sim &
    L_{Y,\phi,\psi}
}\]
is a Hodge isometry.
Let us define
\[
    \widetilde{\mathrm{FM}}(X)\colonequals
    \{
        (Y,\phi,\psi)
        \mid
        Y\in\mathrm{FM}(X),\,
        \phi\colon S\xrightarrow{\sim}S_Y,\,
        \psi\colon T\xrightarrow{\sim}T_Y
    \}
\]
where $\phi$ and $\psi$ are as in \eqref{eqn:SS_Y} and \eqref{eqn:TT_Y}, respectively. Then the above construction gives a diagram
\[\xymatrix{
    \widetilde{\mathrm{FM}}(X)\ar[d]_-{\Pi}\ar[r]^-{L_\bullet} & \cM_{S,T}\\
    \mathrm{FM}(X)
}\]
where $\Pi(Y,\phi,\psi) = Y$ and the map $L_\bullet$ works as in \eqref{eqn:overlattice}.
Note that the preimage of each $Y\in\mathrm{FM}(X)$ under $\pi$ has the form
\[
    \Pi^{-1}(Y) = \{
    (Y,\phi,\psi),\,
    (Y,-\phi,\psi),\,
    (Y,\phi,-\psi),\,
    (Y,-\phi,-\psi)
    \}.
\]
In particular, we have $|\widetilde{\mathrm{FM}}(X)| = 4|\mathrm{FM}(X)|$.
In Lemmas~\ref{lemma:equalityTFM_MST} and \ref{lemma:MST=4}, we will prove, respectively, that
\[
    |\widetilde{\mathrm{FM}}(X)| = 2|\cM_{S,T}|
\]
and that
\[
    |\cM_{S,T}| =
    \begin{cases}
    2m, & \text{if } d \equiv2\ (\mathrm{mod}\ 6),\\
    m, & \text{if } d \equiv0\ (\mathrm{mod}\ 6) \text{ and }9\nmid d.
    \end{cases}
\]
These imply that
\[
    |\mathrm{FM}(X)|
    = \frac{1}{4}|\widetilde{\mathrm{FM}}(X)|
    = \frac{1}{2}|\cM_{S,T}|
    = \begin{cases}
    m, & \text{if } d \equiv2\ (\mathrm{mod}\ 6),\\
    \frac{m}{2}, & \text{if } d \equiv0\ (\mathrm{mod}\ 6) \text{ and }9\nmid d.
    \end{cases}
\]
\end{proof}

\begin{lemma}
\label{lemma:equalityTFM_MST}
Let us retain the condition of Proposition~\ref{prop:FMPartner} and the notation in its proof.
Then we have
$
    |\widetilde{\mathrm{FM}}(X)| = 2|\cM_{S,T}|.
$
\end{lemma}

\begin{proof}
We will mostly assume that $d \equiv2\ (\mathrm{mod}\ 6)$, and will mention the changes needed for the case $d \equiv0\ (\mathrm{mod}\ 6)$ in Remark~\ref{rmk:case_d=0(mod6)}.
Suppose that $Y$ and $Y'$ are Fourier--Mukai partners of $X$ such that $L_{Y,\phi,\psi} = L_{Y',\phi',\psi'}$. Then $Y$ and $Y'$ are isomorphic by the Torelli theorem \cite{Voi86} and \eqref{eqn:a2InPrim}.
This shows that the map $\Pi$ factors as
\[\xymatrix{
    \widetilde{\mathrm{FM}}(X)\ar[d]_-{\Pi}\ar[r]^-{L_\bullet} & \cM_{S,T}\ar[dl]^-{\Pi'}\\
    \mathrm{FM}(X).
}\]
In particular, $L_\bullet$ maps distinct fibers of $\Pi$ to disjoint subsets of $\cM_{S,T}$. If we can show that $L_{Y,\phi,\psi} = L_{Y,-\phi,-\psi}$ and $L_{Y,\phi,\psi}\neq L_{Y,\phi,-\psi}$, then the image of each fiber $\Pi^{-1}(Y)$ under $L_\bullet$ would consist of two elements, so $L_\bullet$ is $2$-to-$1$. This would imply that
\begin{equation}
\label{eqn:inequalityTFM_MST}
    |\widetilde{\mathrm{FM}}(X)| \leq 2|\cM_{S,T}|
\end{equation}
Notice that the equality on the left is trivial since
\[
    L_{Y,-\phi,-\psi} = -L_{Y,\phi,\psi} = L_{Y,\phi,\psi}\subseteq S^*\oplus T^*.
\]

Now we prove the inequality $L_{Y,\phi,\psi}\neq L_{Y,\phi,-\psi}$. Let $L$ be any element in $\cM_{S,T}$. By definition, we have $[L^*:L]=3$ and
\[
    S\oplus T\subseteq L\subseteq L^*\subseteq S^*\oplus T^*.
\]
Using the facts that $[S^*:S]=3d$ and $[T^*:T]=d$, we obtain
\[
    [L:S\oplus T] = [S^*\oplus T^*:L^*]=d.
\]
Since $S\subseteq L$ is saturated, the natural map
\[\xymatrix{
    L^*/(S\oplus T)\ar[r] & S^*/S\cong\bZ/(3d)\bZ
}\]
is a surjection, therefore an isomorphism as $[L^*:S\oplus T]=3d$.
This implies that $L^*/(S\oplus T)$ is cyclic of order $3d$.
Since $T\subseteq L$ is saturated, the map
\[\xymatrix{
    L^*/(S\oplus T)\ar[r] & T^*/T\cong\bZ/d\bZ
}\]
is surjective as well. Write
\[
    S^*/S = \left<\frac{\ell}{3d}\right>
    \quad\text{and}\quad
    T^*/T=\left<\frac{t}{d}\right>
\]
for some $t\in T$. Then there exists an integer $b$ with $\gcd(b, d)=1$ such that
\[
    L^*/(S\oplus T)
    =\left<\frac{\ell}{3d} + \frac{bt}{d}\right>
\]
Thus we can write
\begin{equation}
\label{eqn:generatorForLoverST}
    L/(S\oplus T)
    =\left<\frac{\ell+3bt}{d}\right>.
\end{equation}
Now express $L_{Y,\phi,\psi}/(S\oplus T)$ in the form \eqref{eqn:generatorForLoverST}. Then we have
\[
    L_{Y,\phi,-\psi}/(S\oplus T)
    =\left<\frac{\ell-3bt}{d}\right>.
\]
It follows that $L_{Y,\phi,\psi}=L_{Y,\phi,-\psi}$ if and only if $6b\equiv0\ (\mathrm{mod}\ d)$. This is impossible since $\gcd{(b, 20)}=1$ and $d\equiv2\ (\mathrm{mod}\ 6)$, so we conclude that $L_{Y,\phi,\psi}\neq L_{Y,\phi,-\psi}$.
This finishes the proof of the inequality \eqref{eqn:inequalityTFM_MST}.

To prove the desired equality, it suffices to show that the map
\[\xymatrix{
    L_\bullet\colon\widetilde{\mathrm{FM}}(X)\ar[r] & \cM_{S,T}
}\]
is surjective. Let $I_{21,2}\colonequals\langle1\rangle^{\oplus 21}\oplus\langle-1\rangle^{\oplus 2}$ be the abstract lattice isometric to the middle cohomology of a cubic fourfold and let $h^2\in I_{21,2}$ be a class with $(h^2,h^2)=3$.
By \cite{Nik79}*{Corollary 1.13.3}, the sublattice
$
    \left< h^2 \right>^\perp \subseteq I_{21,2}
$
is the unique even lattice with signature $(20,2)$ and discriminant $3$ up to lattice isomorphism.
Hence, for any element $L\in\cM_{S,T}$, we have
\[
    L(-1) \cong \left< h^2 \right>^\perp\subseteq I_{21,2}
\]
as abstract lattices. Now consider $T(-1)$ as a sublattice of $I_{21,2}$ using the above isomorphism. Then its orthogonal complement
$
    T(-1)^\perp\subseteq I_{21,2}
$
is a rank $2$ lattice of discriminant $d$ and contains $h^2$. 
If $d \equiv2\ (\mathrm{mod}\ 6)$, then
\[
T(-1)^\perp\cong
\begin{pmatrix}
    3 & 1 \\
    1 & \frac{d+1}{3}
\end{pmatrix}.
\]
If $d \equiv0\ (\mathrm{mod}\ 6)$, then
\[
T(-1)^\perp\cong
\begin{pmatrix}
    3 & 0 \\
    0 & \frac{d}{3}
\end{pmatrix}.
\]
One can check by direct computations that such a $T(-1)^\perp$ does not admit labellings with discriminant $2$ or $6$.
By \cite{Laz10}*{Theorem~1.1}, there exists a cubic fourfold $Y$ with a Hodge isometry
\[\xymatrix{
    \eta\colon
    L(-1)
    \ar[r]^-\sim &
    H^4(Y,\bZ)_\mathrm{prim}
}\]
which maps $(T(-1)_\bC)^{2,0}$ to $H^{3,1}(Y,\bC)$.
By the proof of Lemma~\ref{lemma:isometryT(X)}, the lattice $T(-1)$ (respectively, $T(Y)$) does not contain a proper saturated Hodge sublattice that contains $(T(-1)_\bC)^{2,0}$ (respectively, $H^{3,1}(Y,\bC)$).
Therefore, we have
$
    \eta(T(-1))=T(Y).
$
Hence $Y$ is a Fourier--Mukai partner of $X$, and the restrictions of $\eta$ to $S(-1)$ and $T(-1)$ induce a triple $(Y,\phi,\psi)\in\widetilde{\mathrm{FM}}(X)$ such that
$
    L_\bullet(Y,\phi,\psi)=L.
$
This proves the surjectivity of $L_\bullet$.
\end{proof}

\begin{lemma}
\label{lemma:MST=4}
Let us retain the condition of Proposition~\ref{prop:FMPartner} and the notation in its proof. Then we have
\[
    |\cM_{S,T}| =
    \begin{cases}
    2m, & \text{if } d \equiv2\ (\mathrm{mod}\ 6),\\
    m, & \text{if } d \equiv0\ (\mathrm{mod}\ 6) \text{ and not divisible by }9.
    \end{cases}
\]
\end{lemma}

\begin{proof}
We continue assuming $d \equiv2\ (\mathrm{mod}\ 6)$ and will mention the changes needed for the case $d \equiv0\ (\mathrm{mod}\ 6)$ in Remark~\ref{rmk:case_d=0(mod6)}.
From the proof of Lemma~\ref{lemma:equalityTFM_MST}, we know that each $L\in\cM_{S,T}$ satisfies \eqref{eqn:generatorForLoverST}. We claim that the integer $b$ is uniquely determined as an element of $\bZ/d\bZ$. Indeed, if there is another integer $b'$ such that
\[
    L/(S\oplus T)
    =\left<\frac{\ell+3bt}{d}\right>
    =\left<\frac{\ell+3b't}{d}\right>,
\]
then $3(b-b')\equiv0\ (\mathrm{mod}\ d)$ and thus $b\equiv b'\ (\mathrm{mod}\ d)$.
Since $b$ generates $\bZ/d\bZ$, this determines a map
\begin{equation}
\label{eqn:MST->Z/20Z}
\xymatrix{
    \cM_{S,T}\ar[r] & (\bZ/d\bZ)^*
    :L\mapsto\overline{b}
}
\end{equation}
Suppose that $b$ is an integer such that $\overline{b}\in\bZ/d\bZ$ lies in the image of \eqref{eqn:MST->Z/20Z}; that is, there exists an $L\in\cM_{S,T}$ such that \eqref{eqn:generatorForLoverST} holds. Then $L$ is uniquely determined by
\[
    L = S+T+\left<\frac{\ell+3bt}{d}\right>
    \subseteq S^*\oplus T^*.
\]
Hence \eqref{eqn:MST->Z/20Z} is an injection.
Moreover, if an integral overlattice $L\supseteq S\oplus T$ satisfies \eqref{eqn:generatorForLoverST} with $\gcd{(b,d)}=1$, then $L$ has discriminant $3$ and both $S$ and $T$ are saturated in $L$.
As a consequence, the cardinality $|\cM_{S,T}|$ is the same as the number of $\overline{b}\in(\bZ/d\bZ)^*$ such that the overlattice
\[
    S\oplus T\subseteq
    S+T+\left<\frac{\ell+3bt}{d}\right>
    \subseteq S^*\oplus T^*
\]
is even. As $S$ and $T$ are both even, this is equivalent to
\[
    \left(\frac{\ell+3bt}{d}\right)^2
    = \frac{-3d+9b^2t^2}{d^2}
    \in 2\bZ.
\]
This implies that $t^2=cd$ for some integer $c$. Substituting this back into the relation above, we translate it into the equivalent form
$
    3b^2c \equiv1\ \mathrm{mod}\ 2d.
$
Note that the set
\[
    B_c\colonequals\{
        b\in(\bZ/d\bZ)^*:3b^2c \equiv1\ \mathrm{mod}\ 2d
    \}
\]
is nonempty since $\cM_{S,T}\neq\emptyset$.
The proof of \cite{Ogu02}*{Lemma~4.5} shows that if $c$ is an integer such that $B_c\neq\emptyset$, then the cardinality of $B_c$ is $2m$.
This finishes the proof.
\end{proof}

\begin{rmk}
\label{rmk:case_d=0(mod6)}
The proof for the case $d \equiv0\ (\mathrm{mod}\ 6)$ and not divisible by $9$ is essentially the same. In this case, we have $\ell^2=-\frac{1}{3}d$,
\[
S^*/S = \left<\frac{\ell}{d/3}\right>
\quad\text{and}\quad
T^*/T=\left<\frac{t}{d}\right>.
\]
For each $L\in\cM_{S,T}$, there is a unique $\overline{b}\in\big(\bZ/(\frac{d}{3})\bZ\big)^*$ such that
\[
    L = S+T+\left<\frac{3b\ell+t}{d}\right>
    \subseteq S^*\oplus T^*.
\]
Moreover, the cardinality of $\cM_{S,T}$ is the number of elements $\overline{b}\in\big(\bZ/(\frac{d}{3})\bZ\big)^*$
such that the overlattice
\[
    S\oplus T\subseteq
    S+T+\left<\frac{3b\ell+t}{d}\right>
    \subseteq S^*\oplus T^*
\]
is even, which is equivalent to
\[
    \left(\frac{3b\ell+t}{d}\right)^2
    = \frac{-3b^2d+t^2}{d^2}
    \in 2\bZ.
\]
This implies that $t^2=3cd$ for some integer $c$.
Substituting this back into the relation above, one gets
\[
b^2 \equiv c\ \mathrm{mod}\ \frac{2d}{3}.
\]
Since $\gcd(b, \frac{d}{3})=1$ and $d$ is divisible by $6$, we have $\gcd(b, \frac{2d}{3})=1$.
Hence $c$ is an integer such that $\gcd(c, \frac{d}{3})=1$ and the set
\[
    B_c\colonequals\left\{
        b\in\Big(\bZ/(\frac d3)\bZ\Big)^*:b^2 \equiv c\ \mathrm{mod}\ \frac{2d}{3}
    \right\}
\]
is nonempty.
Again using the proof of \cite{Ogu02}*{Lemma~4.5} and the fact that $d$ is not divisible by $9$, one can show that the cardinality of $B_c$ is $m$ if $B_c$ is nonempty.
\end{rmk}

\section{Birational involution on the Veronese locus}
\label{sect:involVeroLocus}

We prove our main theorems in this section, where the core machinery is the Cremona transformation of $\bP^5$ defined by the system of quadrics passing through the Veronese surface $V\subseteq\bP^5$.
We begin with the study of this map, especially on how it induces a birational involution $\sigma_V$ on the Veronese locus $\cC_{20}$.
Then we study its restriction to a cubic fourfold $X\supseteq V$ and prove that $\sigma_V$ realizes Fourier--Mukai partners.
Finally, we analyze how $\sigma_V$ acts on the loci in $\cC_{20}$ known to parametrize rational cubics and prove that new rational cubic fourfolds arise this way.

\subsection{Cremona transform defined by the Veronese surface}
\label{subsect:veroCremona}

Let $V\subseteq\bP^5$ be a Veronese surface, and let $I_V$ be its defining ideal.
According to \cite{CK89}*{Theorem~3.3}, the linear system $|I_V(2)|$ defines a birational map
\[\xymatrix{
    F_V\colon\bP^5\ar@{-->}[r]^-\sim & |I_V(2)|^\vee\cong\bP^5
}\]
such that the inverse $F_V^{-1}$ is also determined by a Veronese surface $V'$. We will assume that $F_V$ is an involution throughout the section. Note that this means $F_V = F_V^{-1}$, which implies that
$$
    V = \mathrm{Bs}(F_V)
    = \mathrm{Bs}(F_V^{-1})
    = V'\subseteq\bP^5
$$
and vice versa. This condition may not hold in general, but we can always achieve it by choosing a $\PGL(6,\bC)$-action to identify $V$ and $V'$.

Let us study how the map $F_V$ acts on the cubics containing $V$.
First of all, we can resolve the indeterminacy of $F_V$ by a single blowup \cite{CK89}. More precisely, the blowups $\mathrm{Bl}_V(\bP^5)$ and $\mathrm{Bl}_{V'}(\bP^5)$ can be canonically identified as the graph
\[
    \Gamma\colonequals\mathrm{graph}(F_V)\subseteq\bP^5\times\bP^5
\]
such that the projections $p$ and $p'$ onto the two copies of $\bP^5$ give the blowups of $\bP^5$ along $V$ and $V'$, respectively.
Together with $F_V$, these form a commutative diagram
\begin{equation}
\label{eqn:blowGamma}
\vcenter{\vbox{
\xymatrix{
	& \Gamma\ar[dl]_p\ar[dr]^{p'} &\\
	\bP^5\ar@{-->}[rr]^{F_V}_\sim && \bP^5.
}}}
\end{equation}

By applying the blowup formula to $p$, we conclude that the Picard group of $\Gamma$ has rank $2$, and is generated by the classes
\begin{itemize}
\setlength\itemsep{0pt}
    \item[] $H$: pullback of the hyperplane class on $\bP^5$ under $p$,
    \item[] $E$: the exceptional class of $p$.
\end{itemize}
Similarly, applying the blowup formula to $p'$ implies that $\Pic(\Gamma)$ is also generated by
\begin{itemize}
\setlength\itemsep{0pt}
    \item[] $H'$: pullback of the hyperplane class on $\bP^5$ under $p'$,
    \item[] $E'$: the exceptional class of $p'$.
\end{itemize}
The fact that $F_V$ is defined by the quadrics passing through $V$ implies that
\begin{equation}
\label{eqn:quadricsThroughVero}
    H' = 2H - E.
\end{equation}
Since the inverse $F_V^{-1}$ is defined in a similar way, we also have
\begin{equation}
\label{eqn:quadricsThroughVero'}
    H = 2H'-E'.
\end{equation}
Hence $e'=2h'-h=2(2h-e)-h$, and thus
\begin{equation}
\label{eqn:cubicSecVero}
    E' = 3H - 2E.
\end{equation}
Equations \eqref{eqn:quadricsThroughVero} and \eqref{eqn:cubicSecVero} provide the transformation rules between the two bases for $\Pic(\Gamma)$ induced by $p$ and $p'$. Moreover, \eqref{eqn:cubicSecVero} reflects the fact that the secant variety of $V$, which is projectively equivalent to the cubic defined by the determinant of matrix~\eqref{eqn:deterVero}, is contracted by $F_V$ onto $V'$.

\begin{prop}
\label{prop:birationalInvolC20}
The map $F_V$ induces a birational involution
\[\xymatrix{
    \sigma_V\colon\cC_{20}\ar@{-->}[r]^-\sim & \cC_{20}
}\]
by taking a cubic $X\supseteq V$ to its proper image $F_V(X)\subseteq\bP^5$. In general, the image $F_V(X)$ for a smooth cubic $X\supseteq V$ is still a cubic containing $V$, though it may be singular.
\end{prop}

\begin{proof}
The strict transform on $\Gamma$ of a cubic fourfold $X\supseteq V$ represents the class $3H-E\in\Pic(\Gamma)$. Using \eqref{eqn:quadricsThroughVero} and \eqref{eqn:quadricsThroughVero'}, we can rewrite it as
\[
    3H-E = H+H' = 3H'-E'.
\]
This shows that the proper image $F_V(X)$ is a cubic containing $V$, hence proves the last assertion.
This also induces a rational map
\[\xymatrix{
    \widetilde{\sigma}_V\colon|I_V(3)|\ar@{-->}[r] & |I_V(3)|
}\]
which is birational as it admits an inverse defined by $F_V^{-1}$.

To show that $\widetilde{\sigma}_V$ descends as a birational involution $\sigma_V$ on $\cC_{20}$, it is sufficient to show that it descends to the birational model $[U_{20}/\PGL_3(\bC)]$ introduced in \eqref{eqn:quotientToC20}. The latter is true since the $\PGL_3(\bC)$-action commutes with $F_V$ by the definition of $F_V$, so the proof is completed.
\end{proof}

We will prove that the birational involution in Proposition~\ref{prop:birationalInvolC20} is nontrivial in Section~\ref{subsect:birationalVeroCubics}. In fact, we will show that it realizes pairs of non-isomorphic Fourier--Mukai partners. As a preparation, let us compute the intersection numbers between the classes in $\Pic(\Gamma)$.

\begin{lemma}
\label{lemma:heGamma}
The intersection numbers between $H,E\in\Pic(\Gamma)$ are
\[
	H^5 = 1,
	\quad H^4E = H^3E^2 = 0,
	\quad H^2E^3 = 4,
	\quad HE^4 = 18,
	\quad E^5 = 51.
\]
The same result holds if $H$ and $E$ are replaced by $H'$ and $E'$.
\end{lemma}
\begin{proof}
The equality $H^5 = 1$ follows from the fact that $H$ corresponds to the hyperplane class.
For the other intersection numbers, let us compute them using the Segre class $s(V,\bP^5)$. Under the embedding $i\colon V\hookrightarrow\bP^5$, we have
\[
	s(V,\bP^5)
	= c(N_{V/\bP^5})^{-1}
	= c(V)\cdot i^*c(\bP^5)^{-1}.
\]
Let us denote the fundamental class of $V$ as $1_V$, the canonical class as $K_V$, and the class of a line from the isomorphism $V\cong\bP^2$ as $\ell$.
Then
\[
	c(V)
	= 1_V - K_V + \chi(V)
	= 1_V + 3\ell + 3.
\]
On the other hand, using the relation $i^*H = 2\ell$, we obtain
\[
	i^*c(\bP^5)
	= (1_V + 2\ell)^6
	= 1_V + 12\ell + 60.
\]
It follows that
\begin{align*}
	s(V, \bP^5) &= (1_V + 3\ell + 3)\cdot (1_V + 12\ell + 60)^{-1}\\
	&= (1_V + 3\ell + 3)\cdot (1_V - 12\ell + 84)\\
	&= 1_V - 9\ell + 51.
\end{align*}
As a result,
\[
	H^{5-k}E^k
	= (-1)^{k-1}\int_{V}(2\ell)^{5-k}\cdot(1_V - 9\ell + 51)
	= \left\{\begin{array}{lll}
		0 & \mbox{if} & k=1, 2\\
		4 & \mbox{if} & k=3\\
		18 & \mbox{if} & k=4\\
		51 & \mbox{if} & k=5.
	\end{array}\right.
\]
The intersections between $H'$ and $E'$ are computed in the same way.
\end{proof}

\begin{cor}
The intersection numbers between $H,H'\in\Pic(\Gamma)$ are
\[
    H^5 = {H'}^5 = 1,
    \quad H^4H' = H{H'}^4 = 2,
    \quad H^3{H'}^2 = H^2{H'}^3 = 4.
\]
\end{cor}
\begin{proof}
These intersections can be computed directly from Lemma~\ref{lemma:heGamma} using the relation $H'=2H-E$.
\end{proof}

\begin{rmk}
These numbers record some geometric information about the map $F_V$. For instance, $F_V$ is birational since $\deg(F_V)={H'}^5=1$; a general line $\ell\subseteq\bP^5$ is mapped to a rational curve $F_V(\ell)$ of degree $H^4{H'}=2$, which reflects the fact that $F_V$ is defined by quadrics.
\end{rmk}

\subsection{Restricting the Cremona map to a cubic fourfold}
\label{subsect:birationalVeroCubics}

The purpose of this part is to improve Proposition~\ref{prop:birationalInvolC20} to the following.

\begin{thm}
\label{thm:cremonaFM}
The Cremona map $F_V$ induces a birational involution
\[\xymatrix{
    \sigma_V\colon\cC_{20}\ar@{-->}[r]^-\sim & \cC_{20}
}\]
by taking a cubic $X\supseteq V$ to its proper image $F_V(X)\subseteq\bP^5$. For a very general $X\in\cC_{20}$, the image $X'$ appears as the unique cubic fourfold such that $\cA_X\cong\cA_{X'}$ and $X\ncong X'$.
\end{thm}

Let $X$ be a cubic fourfold containing a Veronese surface $V$.
Then the restriction of $F_V$ to $X$ produces a birational map
\[\xymatrix{
    f_V\colon X\ar@{-->}[r]^-\sim & X'\colonequals F_V(X)
}\]
where $X'$ is again a cubic containing a Veronese surface $V'$. Here we assume that $X$ is general enough so that $X'$ is smooth. Our strategy in proving the main theorem is to compare the Hodge structures of $X$ and $X'$ via the resolution of $f_V$.

We obtain the resolution by taking the restriction of diagram \eqref{eqn:blowGamma}:
\begin{equation}
\label{eqn:blowY}
\vcenter{\vbox{
\xymatrix{
	& Y\ar[dl]_{\pi}\ar[dr]^{\pi'} &\\
	X\ar@{-->}[rr]^{f_X}_\sim && X'
}}}
\end{equation}
where $\pi$ and $\pi'$ are the blowups at $V$ and $V'$, respectively.
Applying the blowup formula to $\pi$ gives a decomposition
\begin{equation}
\label{eqn:hodgeLatticeIsom}
    H^4(Y,\bZ)\cong H^4(X,\bZ)\oplus H^2(V,\bZ)(-1)
\end{equation}
which preserves the lattice and the Hodge structures. The same decomposition holds with $X$ replaced by $X'$ by applying the same formula to $\pi'$. These induce the Hodge isometries between the transcendental lattices
\begin{equation}
\label{eqn:transcendentalIsom}
\xymatrix{
    T(X)\ar[r]^-{\pi^*}_-\sim
    & T(Y)
    & T(X').\ar[l]_-{{\pi'}^*}^-\sim
}
\end{equation}
In particular, this implies that $X$ and $X'$ are Fourier--Mukai partners.

The difficult part is to show that $X$ and $X'$ are not isomorphic. To attain this goal, we study the restriction of \eqref{eqn:hodgeLatticeIsom} to the algebraic parts
\begin{equation}
\label{eqn:algebraicIsom}
    A(Y)\cong A(X)\oplus A(V)(-1).
\end{equation}
Due to this decomposition, $A(Y)$ contains the classes
\begin{itemize}
\setlength\itemsep{0pt}
    \item[] $h^2$: square of the class $h$ of a hyperplane section on $X$,
    \item[] $v$: the class of $V$ on $X$,
    \item[] $\ell$: the class of a line in $V\cong\bP^2$,
\end{itemize}
which have intersection pairings
\begin{equation}
\label{eqn:labellingOnY}
\begin{pmatrix}
    3 & 4 & 0\\
    4 & 12 & 0\\
    0 & 0 & -1
\end{pmatrix}.
\end{equation}
Decomposition \eqref{eqn:algebraicIsom} still holds with $X$ and $V$ replaced by $X'$ and $V'$, respectively; hence $A(Y)$ also contains
\begin{itemize}
\setlength\itemsep{0pt}
    \item[] ${h'}^2$: square of the class $h'$ of a hyperplane section on $X'$,
    \item[] $v'$: the class of $V'$ on $X'$,
    \item[] $\ell'$: the class of a line in $V'\cong\bP^2$,
\end{itemize}
whose intersection pairings coincide with \eqref{eqn:labellingOnY}.
We will also need the classes
\begin{itemize}
\setlength\itemsep{0pt}
    \item[] $e$: the class of the exceptional divisor of $\pi\colon Y\to X$,
    \item[] $e'$: the class of the exceptional divisor of $\pi'\colon Y\to X'$. 
\end{itemize}
Before proving the main theorem, let us establish a number of lemmas which will also be used in Section~\ref{subsect:actionOnRationalCubics}.

\begin{lemma}
\label{lemma:heY}
The intersection numbers between $h$ and $e$ are
\[
	h^4 = 3,\quad h^3e = 0,\quad h^2e^2 = -4,\quad he^3 = -6, \quad e^4 = 3.
\]
The same equalities hold with $h$ and $e$ replaced by $h'$ and $e'$, respectively.
\end{lemma}

\begin{proof}
These numbers can be computed by using Lemma~\ref{lemma:heGamma} and the fact that $Y = 3H - E$ in $\Pic(\Gamma)$. More explicitly, we have
\begin{align*}
h^4     &= (3H - E)H^4 = 3H^5 = 3,\\
h^3e    &= (3H - E)H^3E = 0,\\
h^2e^2  &= (3H - E)H^2E^2 = -H^2E^3 = -4,\\
he^3    &= (3H - E)HE^3 = 3H^2E^3 - HE^4 = 3\cdot 4 - 18 = -6,\\
e^4     &= (3H - E)E^4 = 3HE^4 - E^5 = 3\cdot 18 - 51 = 3.
\end{align*}
The computations of the intersections between $h'$ and $e'$ are the same.
\end{proof}

\begin{lemma}
\label{lemma:veroLine}
The classes $v$ and $\ell$ can be expressed as
\begin{enumerate}[label=\textup{(\arabic*)}]
\setlength\itemsep{0pt}
    \item\label{veroLine:line_hyperExcept}
    $\ell = \frac{1}{2}he$,
    \item\label{veroLine:vero_LineExcept}
    $v = \frac{3}{2}he-e^2 = 3\ell-e^2$.
\end{enumerate}
The same equations hold with $h,e,v,\ell$ replaced by $h',e',v',\ell'$, respectively.
\end{lemma}

\begin{proof}
Let us retain the notation introduced right before Proposition~\ref{prop:birationalInvolC20}, so that $E$ is the exceptional divisor for the blowup $\Gamma = \mathrm{Bl}_V\bP^5\to\bP^5$, and $E|_Y$ is the exceptional divisor for the blowup $Y = \mathrm{Bl}_VX\to X$. To prove item~\ref{veroLine:line_hyperExcept}, let us consider the fiber square
\begin{equation}
\label{eqn:fiberSquare_X}
\xymatrix{
	E|_Y\ar[d]_\eta\ar@{^(->}[r]^j & Y\ar[d]^{\pi}\\
	V\ar@{^(->}[r]^i & X.
}
\end{equation}
Let $\overline{\ell}\in\Pic(V)$ be the class of a line and $\overline{h}\in\Pic(X)$ be the class of a hyperplane section.
Note that $2\overline{\ell} = i^*\overline{h}$ since $i$ is the Veronese embedding. We also have $h = \pi^*\overline{h}$ from the definition of $h$.
Using these relations, we verify that
\[
	2\ell
	= j_*\eta^*(2\overline{\ell})
	= j_*\eta^*(i^*\overline{h})
	= j_*j^*\pi^*\overline{h}
	= j_*j^*h
	= he.
\]
Hence $\ell = \frac{1}{2}he$, which proves item~\ref{veroLine:line_hyperExcept}. The computation for $\ell'$ is the same.

Now we have $\frac{3}{2}he-e^2 = 3\ell-e^2$ by item~\ref{veroLine:line_hyperExcept}. To prove item~\ref{veroLine:vero_LineExcept}, we need to show that this class equals $v$.
Let $\mathcal{Q}$ be the universal quotient bundle on $E|_Y\cong\bP(N_{V/X})$ so that there is an exact sequence
\[\xymatrix{
    0\ar[r]
    & \cO_E(-1)\ar[r]
    & \eta^*N_{V/X}\ar[r]
    & \mathcal{Q}\ar[r]
    & 0.
}\]
By \cite{EH16}*{Theorem~13.14}, diagram \eqref{eqn:fiberSquare_X} induces the split exact sequence of Chow groups
\[\xymatrix{
    0\ar[r]
    & \mathrm{CH}(V)\ar[r]^-{(i_*,\gamma)}
    & \mathrm{CH}(X)\oplus \mathrm{CH}(E|_Y)\ar[r]^-{\pi^*+j_*}
    & \mathrm{CH}(Y)\ar[r]
    & 0
}\]
where $\gamma\colon\mathrm{CH}(V)\to\mathrm{CH}(E|_Y)$ is defined by $\gamma(x)=-c_1(\mathcal{Q})\eta^*(x)$.
The exactness in the middle implies that
$
    v = \pi^*i_*(1_V) = j_*c_1(\mathcal{Q}).
$
On the other hand, a standard computation shows that
$
    c_1(N_{V/X}) = 3\overline{\ell}.
$
Now we combine these to get
\begin{gather*}
    v = j_*c_1(\mathcal{Q})
    = j_*\eta^*c_1(N_{V/X}) - j_*c_1(\cO_E(-1))\\
    = 3j_*\eta^*\overline{\ell} - e^2
    = 3\ell - e^2.
\end{gather*}
This proves item~\ref{veroLine:vero_LineExcept}.
The computation for $v'$ is the same.
\end{proof}

\begin{lemma}
\label{lemma:transRule}
The two sets of vectors $\{h^2,v,\ell\}$ and $\{{h'}^2,v',\ell'\}$ transform into each other in the following ways:
\[
\begin{array}{ccl}
{h'}^2 &=& 4h^2 - v - 5\ell\\
v' &=& v\\
\ell' &=& 3h^2 - v - 4\ell
\end{array}
\qquad
\begin{array}{ccl}
h^2 &=& 4{h'}^2 - v' - 5\ell'\\
v &=& v'\\
\ell &=& 3{h'}^2 - v' - 4\ell'
\end{array}
\]
\end{lemma}

\begin{proof}
These relations can be derived straightforwardly from the equations $h' = 2h - e$ and $e' = 3h - 2e$ and Lemma~\ref{lemma:veroLine}. First we compute ${h'}^2$:
\begin{align*}
    {h'}^2
    &= (2h - e)^2
    = 4h^2 - 4he + e^2\\
    &= 4h^2 - 8\ell + (3\ell - v)
    = 4h^2 - v - 5\ell.
\end{align*}
Next we compute $v'$:
\begin{align*}
    v'
    &= \frac{3}{2}h'e'-{e'}^2
    = \frac{3}{2}(2h-e)(3h-2e)-(3h-2e)^2\\
    &= \frac{3}{2}he - e^2
    = v.
\end{align*}
Finally we compute $\ell'$:
\begin{align*}
    \ell'
    &= \frac{1}{2}h'e'
    = \frac{1}{2}(2h-e)(3h-2e)\\
    &= \frac{1}{2}(6h^2 - 7he + 2e^2)
    = \frac{1}{2}(6h^2 - 14\ell + 2(3\ell-v))\\
    &= 3h^2 - v - 4\ell.
\end{align*}

The inverse transformation can be computed in the same way. Alternatively, one can verify that the transformation matrix
\[
M\colonequals
\begin{pmatrix}
4 & 0 & 3\\
-1 & 1 & -1\\
-5 & 0 & -4
\end{pmatrix}
\]
is involutive; that is, $M^2 = \mathrm{id}$, which implies that the inverse transformation has the same expression as the original one.
\end{proof}

As a preparation for the next lemma, let us recall that, for a very general $X\in\cC_{20}$, the lattice $A(X)$ has Gram matrix
\[
    \begin{pmatrix}
        3 & 4\\
        4 & 12
    \end{pmatrix}.
\]
Moreover, there are isomorphisms
\[\xymatrix{
	\disc{T(X)}\ar[r]^-\sim &
	\disc{A(X)}\ar[r]^-\sim &
	\bZ/20\bZ,
}\]
where the first isomorphism follows from the fact that $H^4(X, \bZ)$ is unimodular, and the second one can be verified directly (cf.~Lemma~\ref{lemma:dTcyclic}).

\begin{lemma}
\label{lemma:times9}
Assume that $X\in\cC_{20}$ is very general.
Then \eqref{eqn:transcendentalIsom} induces an isometry
\[\xymatrix{
    \disc({\pi^*}^{-1}\circ{\pi'}^*)
    \colon
	\disc T(X)\ar[r]^-\sim
	& \disc T(X')
}\]
which acts as the multiplication by $9$ after identifying $\disc T(X)$ and $\disc T(X')$ with $\bZ/20\bZ$ in a canonical way.
\end{lemma}

\begin{proof}
Let us define
$
    \widetilde{A}(X)\colonequals\langle h^2,v,\ell\rangle
$
and
$
    \widetilde{A}(X')\colonequals\langle {h'}^2,v',\ell'\rangle.
$
Since $X$ is very general, the blowup formula induces the Hodge isometries
\begin{gather*}
\xymatrix{
    \widetilde{A}(X)\ar[r]^-\sim
    & A(Y)
    & \widetilde{A}(X')\ar[l]_-\sim
}\\
\xymatrix{
    T(X)\ar[r]^-\sim
    & T(Y)
    & T(X')\ar[l]_-\sim
}
\end{gather*}
Because $A(Y)$ and $T(Y)$ form orthogonal complements in $H^4(Y,\bZ)$, these isometries induce the commutative diagram
\[\xymatrix{
\disc\widetilde{A}(X)\ar[d]\ar[r]^-\alpha
& \disc\widetilde{A}(X')\ar[d]\\
\disc T(X)\ar[r]^-\beta
& \disc T(X')
}\]
where $\beta$ coincides with the isometry in our lemma.
Using this diagram, we translate the problem on how $\beta$ acts to the problem on how $\alpha$ acts.

Let $(e^2)^*\in\disc\widetilde{A}(X)^*$ denote the dual element of $e^2$ under the choice of basis $\{h^2,e^2,\ell\}$. Using Lemmas~\ref{lemma:heY} and \ref{lemma:veroLine}, it is easy to check that
\[
	(e^2)^* = \frac{1}{20}(4h^2 + 3e^2 - 9\ell).
\]
In particular, $(e^2)^*$ provides a generator for $\disc\widetilde{A}(X)\cong\bZ/20\bZ$.
With respect to this choice, we take the element
\[
	({e'}^2)^* = \frac{1}{20}(4{h'}^2 + 3{e'}^2 - 9\ell')
\]
as our generator for $\disc\widetilde{A}(X')$.
To compare $(e^2)^*$ and $({e'}^2)^*$, we first use Lemmas~\ref{lemma:veroLine} and \ref{lemma:transRule} to obtain
\[
    e^2 = 9{h'}^2 + 4{e'}^2 - 24\ell'
\]
and then rewrite $(e^2)^*$ as
\begin{gather*}
\frac{1}{20}(
4(4{h'}^2 + {e'}^2 - 8\ell')
+ 3(9{h'}^2 + 4{e'}^2 - 24\ell')
- 9(3{h'}^2 + {e'}^2 - 7\ell')
)\\
= \frac{1}{20}(16{h'}^2 + 7{e'}^2 - 41\ell').
\end{gather*}
It follows that
\[
	9({e'}^2)^* - (e^2)^*
	= \frac{1}{20}(20{h'}^2 + 20{e'}^2 - 40\ell')
	= {h'}^2 + {e'}^2 - 2\ell'
\]
which is a lattice element. This means that $\alpha$ works by mapping $(e^2)^*$ to $9({e'}^2)^*$, i.e. as the multiplication by $9$, so we finished the proof.
\end{proof}

\begin{proof}[Proof of Theorem~\ref{thm:cremonaFM}]
The map $\sigma_V$ is well defined and is involutive according to Proposition~\ref{prop:birationalInvolC20}.
The remaining thing to prove is the fact that $\sigma_V$ maps a very general member of $\cC_{20}$ to its unique non-isomorphic Fourier--Mukai partner.

Let $X\in\cC_{20}$ be a very general member, and let $X'\colonequals F_V(X)$. Assume, to the contrary, that $\sigma_V$ is the identity. Then there exists a projective isomorphism $g\colon X\xrightarrow{\sim}X'$
which induces a Hodge isometry
\[\xymatrix{
    g^*\colon H^4(X',\bZ)\ar[r]^-\sim & H^4(X,\bZ)
}\]
such that $g^*({h'}^2) = h^2$ and $g^*(v') = v$. Together with the map $f_V$, these produce two Hodge isometries between the transcendental lattices
\[\xymatrix{
    T(X')\ar@<3pt>[r]^-{g^*}\ar@<-3pt>[r]_-{f_V^*} & T(X)
}\]
hence induce two maps between the discriminant groups
\[\xymatrix{
    \disc T(X)\ar@<3pt>[r]^-{\disc g^*}\ar@<-3pt>[r]_-{\disc f_V^*} & \disc T(X').
}\]
It is clear that $\disc g^*$ acts as the multiplication by $1$ from the construction. On the other hand, $\disc f_V^*$ acts as the multiplication by $9$ according to Lemma~\ref{lemma:times9}.
It follows that, as a Hodge isometry acting on $T(X)$, the composition $f_V^*\circ{g^*}^{-1}$ induces an action on $\disc T(X)\cong\bZ/20\bZ$ which is neither the identity nor the rescaling by $-1$. However, this is forbidden by Lemma~\ref{lemma:isometryT(X)}.

As a result, $X$ and $X'$ are not isomorphic. They are Fourier--Mukai partners since their transcendental lattices are isomorphic by (\ref{eqn:transcendentalIsom}). Moreover, $X'$ is the unique such partner by Proposition~\ref{prop:FMPartner}.
\end{proof}

\begin{rmk}\rm
A similar technique appears in \cite{HL18}, where the second author constructs derived equivalences between K3 surfaces of degree $12$ via Cremona transformations of $\bP^4$.
\end{rmk}

\subsection{Actions on the loci of rational cubic fourfolds}
\label{subsect:actionOnRationalCubics}

In this section, we analyze how $\sigma_V$ acts on the codimension $1$ loci in $\cC_{20}$ that are known to parametrize rational cubic fourfolds, and prove that new rational cubic fourfolds arise this way. The main results of this section are Theorems~\ref{thm:newRationalCubics_details} and \ref{thm:rationalbiggerdisc}. In the following, we use the notation $\cC_{d_1,d_2}^{e}$ to denote the component of the intersection $\cC_{d_1}\cap\cC_{d_2}$ such that for a very general $X\in\cC_{d_1,d_2}^{e}$, the Gram matrix of $A(X)$ is $3\times3$ of determinant $e$.

\begin{thm}
\label{thm:newRationalCubics_details}
For each $d = 26, 38, 42$, the birational involution $\sigma_V$ maps a component of $\cC_{20}\cap\cC_d$ birationally onto a component of $\overline{\cC_{20}}\cap\cC_{d'}$, where $d'$ cannot be in the list
\[
    \{2,\, 6,\, 8,\, 14,\, 18,\, 26,\, 38,\, 42\}.
\]
The three components appear as three distinct irreducible divisors $\cC_{20,26}^{173}$, $\cC_{20,38}^{237}$, and $\cC_{20,42}^{277}$ in $\cC_{20}$ whose images under $\sigma_V$ are:
\[\xymatrix{
    \cC_{20}\cap\cC_{26}\supseteq\cC_{20,26}^{173} \ar@{-->}[r]^-\sim &
    \cC_{20,146}^{173}\subseteq\cC_{20}\cap\cC_{146},
}\]
\[\xymatrix{
    \cC_{20}\cap\cC_{38}\supseteq\cC_{20,38}^{237} \ar@{-->}[r]^-\sim &
    \cC_{20,62}^{237}\subseteq\cC_{20}\cap\cC_{62},
}\]
\[\xymatrix{
    \cC_{20}\cap\cC_{42}\supseteq\cC_{20,42}^{277} \ar@{-->}[r]^-\sim &
    \cC_{20,182}^{277}\subseteq\cC_{20}\cap\cC_{182}.
}\]
As a consequence, there exist at least three irreducible divisors in $\cC_{20}$ which parametrize rational cubic fourfolds which were not known before.
\end{thm}

We start with analyzing the algebraic lattices of a very general member $X\in\cC_{20}\cap\cC_d$ and of its proper image $X'=F_V(X)$.
Assume $X\notin\cC_8$. Then $X$ contains a Veronese surface $V\subseteq X$ by Proposition~\ref{prop:containsVeronese}.
From now on, we denote by $\overline{h}\in\Pic(X)$ the class of a hyperplane section on $X$ and denote by $\overline{v}\in A(X)$ the class of the Veronese surface $V\subseteq X$.
Since $X$ is very general, the algebraic lattice $A(X)$ is of rank $3$, so there exists a surface $S\subseteq X$ whose class $\overline{s}\in A(X)$ does not lie in $\langle\overline{h}^2,\overline{v}\rangle$.
From the blowup formula, we have
\[
    A(Y)\cong\left<h^2, v, s, \ell\right>,
\]
where $Y=\mathrm{Bl}_VX$, the classes $h^2, v, s$ are the pullbacks of $\overline{h}^2,\overline{v}, \overline{s}$, respectively, and $\ell$ is induced from a line in $V\cong\bP^2$.
The Gram matrix of these classes is given by
\[
\begin{array}{c|cccc}
 A(Y)   & h^2 & v & s & \ell\\
    \hline
h^2 & 3 & 4 & \deg(S) & 0\\
v    & 4 & 12 & vs & 0\\
s    & \deg(S) & sv & s^2 & 0\\
\ell & 0 & 0 & 0 & -1.
\end{array}
\]

By Lemma~\ref{lemma:transRule}, we also have
\[
    A(Y)\cong\left<h'^2, v', s, \ell'\right>,
\]
where $h'$, $v'$, and $\ell'$ are induced from $X'$ in the same way as how we obtain $h$, $v$, and $\ell$.
The Gram matrix of these classes is
\[
\begin{array}{c|cccc}
A(Y)    & h'^2 & v' & s & \ell'\\
    \hline
h'^2 & 3 & 4 & 4\deg(S)-vs & 0\\
v'    & 4 & 12 & vs & 0\\
s    & 4\deg(S)-vs & sv & s^2 & 3\deg(S) - vs\\
\ell' & 0 & 0 & 3\deg(S) - vs & -1.
\end{array}
\]
Let $s'\colonequals s + (3\deg(S)-vs)\ell'$. Then we have
\[
    A(Y)\cong\left<h'^2, v', s', \ell'\right>,
\]
with Gram matrix
\[
\begin{array}{c|cccc}
A(Y)    & h'^2 & v' & s' & \ell'\\
    \hline
h'^2 & 3 & 4 & 4\deg(S)-vs & 0\\
v'    & 4 & 12 & vs & 0\\
s'    & 4\deg(S)-vs & sv & s^2+(3\deg(S) - vs)^2 & 0\\
\ell' & 0 & 0 & 0 & -1.
\end{array}
\]

By the blowup formula, we have an isometry
\[
A(Y)\cong A(X')\oplus A(V')(-1),
\]
where $A(V')(-1)$ is generated by the class $\ell'$. Therefore, the top-left $3\times3$ minor of the above matrix gives a Gram matrix of $A(X')$.
In this way, we can relate the algebraic lattices $A(X)$ and $A(X')$ as follows:
\begin{equation}
\label{eq:CremonaLatticeTransformation}
\begin{array}{c|ccc}
A(X)    & \overline{h}^2 & \overline{v} & \overline{s} \\
    \hline
\overline{h}^2 & 3 & 4 & A \\
\overline{v}    & 4 & 12 & B \\
\overline{s}    & A & B & C
\end{array}
\mapsto
\begin{array}{c|ccc}
A(X')    & \overline{h}'^2 & \overline{v}' & \overline{s}' \\
    \hline
\overline{h}'^2 & 3 & 4 & 4A-B \\
\overline{v}'    & 4 & 12 & B \\
\overline{s}'    & 4A-B & B & C+(3A - B)^2.
\end{array}
\end{equation}
Here the classes $\overline{h}'^2, \overline{v}', \overline{s}'\in A(X')$ are the images of $h'^2, v', s'\in A(Y)$ under the isometry above.

Note that the transformation law (\ref{eq:CremonaLatticeTransformation}) works for any class $\overline{s}\in A(X)$:
if one replaces $\overline{s}$ with any other class $\overline{s}_0\in A(X)$ such that $A(X)=\left<\overline{h}^2,\overline{v},\overline{s}_0\right>$ with Gram matrix
\[
\begin{array}{c|ccc}
A(X)    & \overline{h}^2 & \overline{v} & \overline{s}_0 \\
    \hline
\overline{h}^2 & 3 & 4 & D \\
\overline{v}    & 4 & 12 & E \\
\overline{s}_0    & D & E & F,
\end{array}
\]
then one can check that the lattice with Gram matrix
\[
\begin{pmatrix}
3 & 4 & 4D-E \\
4 & 12 & E \\
4D-E & E & F+(3D - E)^2 
\end{pmatrix}
\]
is isometric to $\left<\overline{h}'^2,\overline{v}',\overline{s}'\right>$.
Also, one can check that the transformation (\ref{eq:CremonaLatticeTransformation}) preserves discriminants.

Now we recall a useful property of algebraic lattices of cubic fourfolds.
Recall that the middle cohomology of a smooth cubic $X\subseteq\bP^5$ is
\[
H^4(X,\bZ)\cong I_{21,2}\colonequals
E_8^{\oplus2}\oplus U^{\oplus2}\oplus\langle1\rangle^{\oplus3}.
\]
Under this isomorphism, one can identify $A(X)$ with a sublattice in $I_{21,2}$.
Set $h^2\colonequals(1,1,1)\in \langle1\rangle^{\oplus3}\subseteq I_{21,2}$.

\begin{prop}[\cite{YY20}*{Proposition 2.3, Lemma 2.4}, \cite{Has16}*{\S2.3}]
\label{prop:YangYuHassett}
Let $M$ be a positive definite lattice of rank $r$ admitting a saturated embedding
\[
    h^2\in M\subseteq I_{21,2}.
\]
Let $\cC_M\subseteq\cC$ be the subset of cubic fourfolds $X$ with $M\subseteq A(X)\subseteq I_{21,2}$.
Then $\cC_M$ is nonempty if and only if the pairing $(x,x)$ is not $2$ for any $x\in M$.
In this case, $\cC_M\subseteq\cC$ is a codimension $r-1$ subvariety, and there exists an $X\in\cC_M$ with $A(X)=M$.
\end{prop}

Using this criterion, we can prove the following proposition, which is a generalization of \cite{YY20}*{Theorem 3.1}. The proposition will be used to prove the main theorems of this section.

\begin{prop}
\label{prop:existalglattice}
For every nonempty $\cC_{d_1}$ and $\cC_{d_2}$ with $d_1\neq d_2$, we find the following lower bound of the number of irreducible components of $\cC_{d_1}\cap\cC_{d_2}$.
Let $n_1=\lfloor\frac{d_1}{6}\rfloor$ and $n_2=\lfloor\frac{d_2}{6}\rfloor$.
Define
\[
    N=\ceil[\Big]{2\sqrt{n_1n_2-\min\{n_1,n_2\}}-1}.
\]
\begin{enumerate}[label=\textup{(\arabic*)}]
\setlength\itemsep{0pt}
    \item
    Suppose $d_1\equiv d_2\equiv2\ (\mathrm{mod}\ 6)$.
    Then $\cC_{d_1}\cap\cC_{d_2}$ has at least $2N+1$ irreducible components $\cC_{M_\tau}$ where $\tau$ is an integer index with $|\tau|\leq N$. For each $\tau$, there exists an $X\in\cC_{M_\tau}$ such that $A(X)$ is a rank $3$ lattice of discriminant
    \[
    \frac{d_1d_2-(1-3\tau)^2}{3}.
    \]
    \item
    Suppose $d_1d_2\equiv0\ (\mathrm{mod}\ 6)$.
    Then $\cC_{d_1}\cap\cC_{d_2}$ has at least $N+1$ irreducible components $\cC_{M_\tau}$ where $\tau$ is an integer index with $0\leq\tau\leq N$. For each $\tau$, there exists an $X\in\cC_{M_\tau}$ such that $A(X)$ is a rank $3$ lattice of discriminant
    \[
    \frac{d_1d_2}{3}-3\tau^2.
    \]
\end{enumerate}
Moreover, we have $X\notin\cC_8$ in either case if $d_1,d_2\neq8$. (Gram matrices for $A(X)$ in each case are given by \eqref{eqn:2and2_mod6}, \eqref{eqn:2and0_mod6}, and \eqref{eqn:0and0_mod6} in the proof.)
\end{prop}

\begin{proof}
The argument is similar to the proof of \cite{YY20}*{Theorem 3.1}.
Write
\[
I_{21,2}=E_8^{\oplus2}\oplus U_1\oplus U_2\oplus\langle1\rangle^{\oplus3}.
\]
Let $e_i,f_i$ be a basis of $U_i$ such that $e_i^2=f_i^2=0$ and $(e_i,f_i)=1$.
We denote elements in $\langle1\rangle^{\oplus3}$ by a triple of integers $(z_1,z_2,z_3)$.

\medskip
\noindent\textbf{Case 1}: $d_1\equiv d_2\equiv2\ (\mathrm{mod}\ 6)$.
\smallskip

Write $d_1=6n_1+2$ and $d_2=6n_2+2$. For each $|\tau|\leq N$, consider the lattice $M_\tau\subseteq I_{21,2}$ generated by
\begin{align*}
    \alpha_1 & \colonequals h^2 = (1,1,1), \\
    \alpha_2 & \colonequals e_1 + n_1f_1 + \tau f_2 + (0,1,0), \\
    \alpha_3 & \colonequals e_2 + n_2f_2 + (0,0,1).
\end{align*}
Then the Gram matrix of $M_\tau$ with respect to the basis $\{\alpha_1,\alpha_2,\alpha_3\}$ is
\begin{equation}
\label{eqn:2and2_mod6}
\begin{pmatrix}
3 & 1 & 1 \\
1 & 2n_1+1 & \tau \\
1 & \tau & 2n_2+1
\end{pmatrix}.
\end{equation}

We claim that $(x,x)\neq2$ for any $x\in M_\tau$. Write
$
x=a\alpha_1+b\alpha_2+c\alpha_3,
$
and assume without loss of generality that $1\leq n_1<n_2$. Then
\begin{align*}
    (x,x) & = 3a^2 + (2n_1+1)b^2 + (2n_2+1)c^2 + 2ab + 2ac + 2\tau bc \\
    & = a^2+(a+b)^2+(a+c)^2 + 2n_1b^2 + 2n_2c^2 + 2\tau bc \\
    & = a^2+(a+b)^2+(a+c)^2 + 2n_1(b+\frac{\tau}{2n_1}c)^2 + (2n_2-\frac{\tau^2}{2n_1})c^2.
\end{align*}
Since $|\tau|<2\sqrt{n_1n_2-n_1}$ by assumption, we have
$
    2n_2-\frac{\tau^2}{2n_1}>2.
$
Hence if $(x,x)=2$, then $c=0$. Thus
\[
2=(x,x)=2a^2+2n_1b^2+(a+b)^2.
\]
But there are no integers $a,b$ satisfying this equation.

One can check that $M_\tau\subseteq I_{21,2}$ is saturated. Therefore, by Proposition~\ref{prop:YangYuHassett}, the subvariety $\cC_{M_\tau}\subseteq\cC$ has codimension $2$, and there exists an $X\in\cC_{M_\tau}$ with $A(X)=M_\tau$ which has discriminant
\[
\frac{d_1d_2-(1-3\tau)^2}{3}.
\]
Also, observe that the sublattices
\[
h^2\in K_{d_1}\colonequals\langle\alpha_1,\alpha_2\rangle\subseteq M_\tau
\quad\text{and}\quad
h^2\in K_{d_2}\colonequals\langle\alpha_1,\alpha_3\rangle\subseteq M_\tau
\]
are both saturated. Therefore,
$
    \cC_{M_\tau}\subseteq\cC_{d_1}\cap\cC_{d_2}.
$

Now we show that such a cubic fourfold $X$ does not lie in $\cC_8$ if $d_1,d_2\neq8$. Let $2\leq n_1<n_2$. Observe from the Gram matrix of $A(X)=M_\tau$ that it contains a labelling of discriminant $8$ if and only if there exist integers $b$ and $c$ such that
\begin{align*}
    8 & =(6n_1+2)b^2+(6n_2+2)c^2+(6\tau-2)bc \\
    & = (6n_1+2)\Big(b+\frac{3\tau-1}{6n_1+2}c\Big)^2 + \Big(6n_2+2 - \frac{(3\tau-1)^2}{6n_1+2}\Big)c^2
\end{align*}
To show that there are no integer solutions, it suffices to show that
\[
6n_2+2 - \frac{(3\tau-1)^2}{6n_1+2}>8,
\]
or equivalently
\[
6(n_2-1)(6n_1+2)>(3\tau-1)^2.
\]
Recall that $|\tau|<2\sqrt{n_1n_2-n_1} = 2\sqrt{n_1(n_2-1)}$.
Using $n_1\leq n_2-1$, we obtain $|\tau|\leq2n_2-3$. Therefore,
\begin{align*}
    6(n_2-1)(6n_1+2) & = 36(n_2-1)n_1 + 12(n_2-1) \\
    & > 9\tau^2 + 6|\tau| + 1 \geq(3\tau-1)^2.
\end{align*}
This proves $X\notin\cC_8$.

The two remaining cases of possible $d_1,d_2$ can be proved by the same procedure, so we will only list a basis for $M_\tau$ and its Gram matrix and omit the details of the computations.

\medskip
\noindent\textbf{Case 2}: $d_1\equiv2\ (\mathrm{mod}\ 6)$ and $d_2\equiv0\ (\mathrm{mod}\ 6)$.
\smallskip

Write $d_1=6n_1+2$ and $d_2=6n_2$. For each $\tau$, consider the lattice $M_\tau\subseteq I_{21,2}$ generated by
\begin{align*}
    \alpha_1 & \colonequals h^2 = (1,1,1), \\
    \alpha_2 & \colonequals e_1 + n_1f_1 + \tau f_2 + (0,1,0), \\
    \alpha_3 & \colonequals e_2 + n_2f_2.
\end{align*}
The Gram matrix of $M_\tau$ with respect to the basis $\{\alpha_1,\alpha_2,\alpha_3\}$ is
\begin{equation}
\label{eqn:2and0_mod6}
\begin{pmatrix}
3 & 1 & 0 \\
1 & 2n_1+1 & \tau \\
0 & \tau & 2n_2
\end{pmatrix}.
\end{equation}

\medskip
\noindent\textbf{Case 3}: $d_1\equiv d_2\equiv0\ (\mathrm{mod}\ 6)$.
\smallskip

Write $d_1=6n_1$ and $d_2=6n_2$. For each $\tau$, consider the lattice $M_\tau\subseteq I_{21,2}$ generated by
\begin{align*}
    \alpha_1 & \colonequals h^2 = (1,1,1), \\
    \alpha_2 & \colonequals e_1 + n_1f_1 + \tau f_2, \\
    \alpha_3 & \colonequals e_2 + n_2f_2.
\end{align*}
The Gram matrix of $M_\tau$ with respect to the basis $\{\alpha_1,\alpha_2,\alpha_3\}$ is
\begin{equation}
\label{eqn:0and0_mod6}
\begin{pmatrix}
3 & 0 & 0 \\
0 & 2n_1 & \tau \\
0 & \tau & 2n_2
\end{pmatrix}.
\end{equation}
This finishes the proof.
\end{proof}

The following lemma is a simple calculation that will be used several times later on.

\begin{lemma}\label{lemma:disc34412}
    Let $X$ be a cubic fourfold such that $A(X)=\left<h^2,v,s\right>$ has Gram matrix
    $$
    \begin{pmatrix}
        3&4&A \\ 4&12&B \\ A&B&C
    \end{pmatrix}.
    $$
    Then $X\in\cC_d$ implies that there exist $b,c\in\bZ$ such that
    $$
        d=20b^2+(6B-8A)bc+(3C-A^2)c^2.
    $$
\end{lemma}

\begin{proof}
Let $b$ and $c$ be two integers.
The determinant of the Gram matrix of $\{h^2,bv+cs\}$ is given by $20b^2+(6B-8A)bc+(3C-A^2)c^2$.
\end{proof}

We are now ready prove the main theorems of this section.

\begin{proof}[Proof of Theorem~\ref{thm:newRationalCubics_details}]
Consider the intersection $\cC_{20}\cap\cC_{26}$.
By Proposition~\ref{prop:existalglattice} (and \eqref{eqn:2and2_mod6}), there exists a cubic fourfold $X$ such that $A(X)\cong M_0$ has Gram matrix
\[
\begin{pmatrix}
3 & 1 & 1 \\
1 & 7 & 0 \\
1 & 0 & 9
\end{pmatrix},
\text{ or equivalently (as lattices) }
\begin{pmatrix}
3 & 4 & 1 \\
4 & 12 & 1 \\
1 & 1 & 9
\end{pmatrix}
=: A.
\]
Set
$
\cC_{20,26}^{173}\colonequals\cC_{M_0}\subseteq\cC,
$
where $173$ is the determinant of the above matrices.
By Proposition~\ref{prop:existalglattice}, we have
\[
X\in\cC_{20,26}^{173}\subseteq\cC_{20}\cap\cC_{26}
\quad\text{and}\quad
X\notin\cC_8.
\]
Therefore, $X$ contains a Veronese surface $V$ by Proposition~\ref{prop:containsVeronese}. There exists a class $\overline{s}\in A(X)$ such that
$
A(X)=\left<\overline{h}^2, \overline{v}, \overline{s}\right>,
$
and the Gram matrix with respect to this basis is $A$.
By the transformation law \eqref{eq:CremonaLatticeTransformation}, the algebraic lattice of $X'\colonequals F_V(X)$ is isometric to
\[
\begin{pmatrix}
3 & 4 & 3 \\
4 & 12 & 1 \\
3 & 1 & 13
\end{pmatrix},
\text{ or equivalently }
\begin{pmatrix}
3 & 1 & 1 \\
1 & 7 & -16 \\
1 & -16 & 49
\end{pmatrix}
=: A'.
\]
By Proposition~\ref{prop:existalglattice} (and \eqref{eqn:2and2_mod6}) and also the fact that
\[
    |-16|\leq\ceil[\Big]{2\sqrt{3\cdot24-3} - 1}
\]
the cubic $X'$ lies in a codimension two irreducible component
\[
X'\in\cC_{20,146}^{173}\subseteq\cC_{20}\cap\cC_{146}.
\]
By Lemma~\ref{lemma:disc34412}, we have $X'\in\cC_{d'}$ only if
$
d'=20b^2-18bc+30c^2
$
for some $b,c\in\bZ$.
It is not hard to verify that
\[
    \{2,\, 6,\, 8,\, 14,\, 18,\, 26,\, 38,\, 42\}
    \cap
    \{20b^2-18bc+30c^2:b,c\in\bZ\}
    =\emptyset.
\]
For example, to rule out the case $d'=2$, consider the equations
\begin{align*}
    2 &= 20b^2-18bc+30c^2\\
    &= 20\left(b-\frac{9}{20}c\right)^2+\left(30-\frac{81}{20}\right)c^2.
\end{align*}
This forces $c=0$. By substituting this back, we get $2=20b^2$, which has no integer solution. The other cases can be verified in the same way. (As pointed out by the anonymous referee, a simpler argument is to use the fact that the smallest integers that can be represented by a \emph{reduced} positive definite binary form $Ax^2+Bxy+Cy^2$ are $0,A,C,A-|B|+C,A+|B|+C$, which in our case are $0,20,30,32,68$.)
This proves that $X'\in\cC_{20}\cap\cC_{146}$ is a rational cubic fourfold not known before. 

We claim that the map between the codimension $2$ components
\begin{equation}
\label{eqn:(20,26)-(20,146)}
\xymatrix{
    \cC_{20}\cap\cC_{26}\supseteq\cC_{20,26}^{173} \ar@{-->}[r]^{\sigma_V} &
    \cC_{20,146}^{173}\subseteq\cC_{20}\cap\cC_{146}
}
\end{equation}
is birational. First we show that it is dominant. Assume, to the contrary, that there exist infinitely many $X\in\cC_{20,26}^{173}$ mapped to the same point in $\cC_{20,146}^{173}$. Then the transcendental lattices $T(X)$ of these cubics are all isometric via resolutions as in \eqref{eqn:blowY} and the blowup formula. By \cite{AT14}*{Theorem~3.1}, the lattice $N(\cA_X)$ contains a copy of the hyperbolic plane
\[
    U=\begin{pmatrix}0&1\\1&0\end{pmatrix}.
\]
Hence the isometries among the $T(X)$ can be extended to isometries among the $\mukaiH(\cA_X,\bZ)$ by \cite{Nik79}*{Theorem 1.14.4}.
This gives a contradiction since for any cubic fourfold $X$, there are only finitely many cubic fourfolds $X'$ such that there is a Hodge isometry $\mukaiH(\cA_X,\bZ)\cong\mukaiH(\cA_{X'},\bZ)$; see \cite{Huy17}*{Corollary~3.5}. This proves that \eqref{eqn:(20,26)-(20,146)} is quasi-finite and thus is dominant. Now, the conclusion of the previous paragraph implies that a very general $X'\in\cC_{20,146}^{173}$ lies outside $\cC_8$, whence $\sigma_V$ is well defined at $X'$ by Proposition~\ref{prop:containsVeronese}. Because $\sigma_V$ is an involution, we get $\sigma_V^{-1}(X') = \sigma_V(X')$, which shows that the preimage of $X'$ under $\sigma_V$ consists of only one element. This proves that \eqref{eqn:(20,26)-(20,146)} is birational. As a consequence, $\cC_{20,146}^{173}$ contains a Zariski open subset that parametrizes rational cubics, which implies that all cubics in $\cC_{20,146}^{173}$ are rational by \cite{KT19}*{Theorem~1}.

We can apply the same argument to the intersections $\cC_{20}\cap\cC_{38}$ and $\cC_{20}\cap\cC_{42}$ to find new rational cubic fourfolds.
For $\cC_{20}\cap\cC_{38}$, we start with a cubic fourfold
$
X\in\cC_{20,38}^{237}\subseteq\cC_{20}\cap\cC_{38}
$
whose Gram matrix of $A(X)$ is
\[
\begin{pmatrix}
3 & 1 & 1 \\
1 & 7 & -2 \\
1 & -2 & 13
\end{pmatrix},
\text{ or equivalently }
\begin{pmatrix}
3 & 4 & 1 \\
4 & 12 & -1 \\
1 & -1 & 13
\end{pmatrix}.
\]
After Cremona transformation, the Gram matrix of $A(X')$ of the proper image $X'$ is
\[
\begin{pmatrix}
3 & 4 & 5 \\
4 & 12 & -1 \\
5 & -1 & 29
\end{pmatrix},
\text{ or equivalently }
\begin{pmatrix}
3 & 1 & 1 \\
1 & 7 & 8 \\
1 & 8 & 21
\end{pmatrix}.
\]
A straightforward computation shows that
$
X'\in\cC_{20,62}^{237}\subseteq\cC_{20}\cap\cC_{62}
$
and $X'\notin\cC_d'$ for any $d'\in\{2,\, 6,\, 8,\, 14,\, 18,\, 26,\, 38,\, 42\}$.

For $\cC_{20}\cap\cC_{42}$, we start with a cubic fourfold
$
X\in\cC_{20,42}^{277}\subseteq\cC_{20}\cap\cC_{42}
$
whose Gram matrix of $A(X)$ is
\[
\begin{pmatrix}
3 & 1 & 0 \\
1 & 7 & 1 \\
0 & 1 & 14
\end{pmatrix},
\text{ or equivalently }
\begin{pmatrix}
3 & 4 & 0 \\
4 & 12 & 1 \\
0 & 1 & 14
\end{pmatrix}.
\]
After Cremona transformation, the Gram matrix of $A(X')$ of the proper image $X'$ is
\[
\begin{pmatrix}
3 & 4 & -1 \\
4 & 12 & 1 \\
-1 & 1 & 15
\end{pmatrix},
\text{ or equivalently }
\begin{pmatrix}
3 & 1 & 1 \\
1 & 7 & 18 \\
1 & 18 & 61
\end{pmatrix}.
\]
A straightforward computation shows that
$
X'\in\cC_{20,182}^{277}\subseteq\cC_{20}\cap\cC_{182}
$
and $X'\notin\cC_d'$ for any $d'\in\{2,\, 6,\, 8,\, 14,\, 18,\, 26,\, 38,\, 42\}$.
\end{proof}

\begin{rmk}
One can find the images of the other components of
\[
\cC_{20}\cap\cC_{26},
\quad\cC_{20}\cap\cC_{38},
\quad\cC_{20}\cap\cC_{42}
\]
that we found in Proposition \ref{prop:existalglattice} under $\sigma_V$ using the same argument.
It turns out that most of these components are invariant under the action of $\sigma_V$.
Exceptions include the three components we found in Theorem~\ref{thm:newRationalCubics_details} and also the following:
\[\xymatrix{
    \cC_{20}\cap\cC_{26}\supseteq\cC_{20,26}^{\tau=4} \ar@{-->}[r]^-\sim &
    \cC_{20,38}^{\tau=-6}=\cC_{20,42}^{\tau=7}\subseteq\cC_{20}\cap\cC_{38}\cap\cC_{42},
}\]
\[\xymatrix{
    \cC_{20}\cap\cC_{26}\supseteq\cC_{20,26}^{\tau=-2} \ar@{-->}[r]^-\sim &
    \cC_{20,38}^{\tau=6}\subseteq\cC_{20}\cap\cC_{38},
}\]
\[\xymatrix{
    \cC_{20}\cap\cC_{38}\supseteq\cC_{20,38}^{\tau=0} \ar@{-->}[r]^-\sim &
    \cC_{20,42}^{\tau=3}\subseteq\cC_{20}\cap\cC_{42}.
}\]
Here $\tau$ is the parameter of irreducible components in $\cC_{20}\cap\cC_{d}$ we used in Proposition \ref{prop:existalglattice}.
Note that these additional codimension $2$ loci that $\sigma_V$ moves do not provide new rational cubic fourfolds.
\end{rmk}

\begin{rmk}
\label{rmk:C20C14}
One can prove that the intersection $\cC_{20}\cap\cC_{14}$ has nine irreducible components $\cC_{M_\tau}$ by the criterion in Proposition \ref{prop:YangYuHassett}, where $\tau$ is an integer index with $|\tau|\leq4$. The Gram matrix of the algebraic lattice of a general cubic in $\cC_{M_\tau}$ is
\[
\begin{pmatrix}
    3 & 1 & 1 \\
    1 & 7 & \tau \\
    1 & \tau & 5
\end{pmatrix},
\text{ or equivalently }
\begin{pmatrix}
3 & 4 & 1 \\
4 & 12 & \tau+1 \\
1 & \tau+1 & 5
\end{pmatrix}.
\]
Note that Proposition \ref{prop:existalglattice} only guarantees the existence of irreducible components for $|\tau|\leq3$.
Under the action of $\sigma_V$, six of the nine components are mapped to $\cC_{14}$, $\cC_{26}$, $\cC_{38}$, or $\cC_{42}$.
The remaining three are $\cC_{M_0}$, $\cC_{M_{4}}$, and $\cC_{M_{-4}}$. The map $\sigma_V$ acts on $\cC_{M_0} = \cC_{20,14}^{\tau=0}$ as
\[\xymatrix{
    \cC_{20}\cap\cC_{14}\supseteq\cC_{20,14}^{\tau=0} \ar@{-->}[r]^-\sim &
    \cC_{20,62}^{\tau=-10}=\cC_{20,18}^{\tau=3}\subseteq\cC_{20}\cap\cC_{62}\cap\cC_{18}.
}\]
In this case, the image component belongs to the list of infinitely many divisors in $\cC_{18}$ found in \cite{AHTV19}.
The other two components $\cC_{M_{-4}}$ and $\cC_{M_{4}}$ are contained in $\cC_8$ and $\cC_{12}$ respectively. Suppose that a general member of these components contains a Veronese surface, then $\sigma_V$ acts on them as
\[\xymatrix{
    \cC_{20}\cap\cC_{14}\supseteq\cC_{20,14}^{\tau=4} \ar@{-->}[r]^-\sim &
    \cC_{20,26}^{\tau=-6}\subseteq\cC_{20}\cap\cC_{26},
}\]
\[\xymatrix{
    \cC_{20}\cap\cC_{14}\supseteq\cC_{20,14}^{\tau=-4} \ar@{-->}[r]^-\sim &
    \cC_{20,6}^{\tau=1}\subseteq\overline{\cC_{20}}\cap\cC_{6},
}\]
Note that the image cubics in the second case would be singular. 
\end{rmk}

Finally, given an admissible $d\geq 14$ with $d\equiv2\ (\mathrm{mod}\ 6)$, we prove the result below by using the component of $\cC_{20}\cap\cC_d$ marked by Gram matrix \eqref{eqn:2and2_mod6} with $\tau=0$. In the case that $d\equiv0\ (\mathrm{mod}\ 6)$, we obtain the same result by using the component marked by Gram matrix \eqref{eqn:2and0_mod6} with $\tau=1$.

\begin{thm}
\label{thm:rationalbiggerdisc}
Let $d\geq 14$ be an even integer which is admissible, i.e. satisfies \eqref{eqn:discriminantK3}.
Then $\sigma_V(\cC_{20}\cap\cC_d)$ contains a component $D$ such that
\begin{enumerate}[label=\textup{(\arabic*)}]
\setlength\itemsep{0pt}
    \item $D\not\subseteq\cC_{d'}$ for any admissible $d'$ with $d'\leq d$.
    \item $D\subseteq\cC_{d'}$ for some admissible $d'$ with $d'>d$.
\end{enumerate}
\end{thm}

\begin{proof}
Let $d\geq 14$ be an integer in the list \eqref{eqn:discriminantK3}.
First, we claim that there exists a component $D\subseteq\sigma_V(\cC_{20}\cap\cC_d)$ such that a very general cubic $X'\in D$ has the following property:
\[
\text{If } X'\in\cC_{d'} \text{ and } d'\leq d, \text{ then } 20\mid d'.
\]
The first part of Theorem~\ref{thm:rationalbiggerdisc} then follows from the claim because then any such $d'$ would not satisfy (\ref{eqn:discriminantK3}).

\medskip
\noindent\textbf{Case 1}: $d\equiv2\ (\mathrm{mod}\ 6)$.
\smallskip

Write $d=6n+2$. By Proposition~\ref{prop:existalglattice}, there exists a component of $\cC_{20}\cap\cC_{d}$ such that the Gram matrix of $A(X)$ of a general cubic $X$ in the component is given by
\[
\begin{pmatrix}
3 & 1 & 1 \\
1 & 7 & 0 \\
1 & 0 & 2n+1
\end{pmatrix},
\text{ or equivalently }
\begin{pmatrix}
3 & 4 & 1 \\
4 & 12 & 1 \\
1 & 1 & 2n+1
\end{pmatrix}.
\]
Following the proof of Theorem \ref{thm:newRationalCubics_details}, the map $\sigma_V$ sends this component birationally to a component $D\subseteq\sigma_V(\cC_{20}\cap\cC_d)$.
The Gram matrix of $A(X')$ of a general cubic $X'\in D$ is given by
\[
\begin{pmatrix}
3 & 4 & 3 \\
4 & 12 & 1 \\
3 & 1 & 2n+5
\end{pmatrix}.
\]
By Lemma~\ref{lemma:disc34412}, the cubic $X'$ is in $\cC_{d'}$ only if
$
d'=20b^2-18bc+(6n+6)c^2
$
for some $b,c\in\bZ$.
One can check that there are no such integers $b,c$ satisfying the equation if $d'\leq d$ and $20\nmid d'$.

\medskip
\noindent\textbf{Case 2}: $d\equiv0\ (\mathrm{mod}\ 6)$.
\smallskip

Write $d=6n$. By Proposition~\ref{prop:existalglattice}, there exists a component of $\cC_{20}\cap\cC_{d}$ such that the Gram matrix of $A(X)$ of a very general cubic $X$ in this component is
\[
\begin{pmatrix}
3 & 1 & 0 \\
1 & 7 & 1 \\
0 & 1 & 2n
\end{pmatrix},
\text{ or equivalently }
\begin{pmatrix}
3 & 4 & 0 \\
4 & 12 & 1 \\
0 & 1 & 2n
\end{pmatrix}.
\]
The map $\sigma_V$ sends this component birationally onto $D\subseteq\sigma_V(\cC_{20}\cap\cC_d)$.
The Gram matrix of $A(X')$ of a general $X'\in D$ is
\[
\begin{pmatrix}
3 & 4 & -1 \\
4 & 12 & 1 \\
-1 & 1 & 2n+1
\end{pmatrix}.
\]
By Lemma~\ref{lemma:disc34412}, the cubic $X'$ is in $\cC_{d'}$ only if
$
d'=20b^2+14bc+(6n+2)c^2
$
for some $b,c\in\bZ$.
One can check that there are no such integers $b,c$ satisfying the equation if $d'\leq d$ and $20\nmid d'$.
This concludes the proof of the claim. Therefore, the first part of Theorem~\ref{thm:rationalbiggerdisc} holds.

Next, we show that $D\subseteq\cC_{d'}$ for some admissible $d'$.
Assume, to the contrary, that $D\not\subseteq\cC_{d'}$ for any such $d'$. Then
$
\bigcup_{d':\text{ admissible}}(D\cap\cC_{d'})
$
is a proper subset of $D$, so a very general member of $D$ does not lie in $\cC_{d'}$ for any admissible $d'$.
Therefore, to prove the second part of Theorem \ref{thm:rationalbiggerdisc}, it suffices to show that a very general member of $D$ lies in $\cC_{d'}$ for some admissible $d'$.

Let $X'=\sigma_V(X)$ be a very general cubic in $D$, where $X\in\cC_{20}\cap\cC_{d}$ for some admissible $d$.
By the blowup formula, we have $T(X) \cong T(X')$ and thus
$
T(\cA_X)\cong T(\cA_{X'}).
$
By \cite{AT14}*{Theorem~3.1 \& Proposition~2.4}, since $d$ is admissible, the lattice $N(\cA_X)$ contains a copy of the hyperbolic plane
\[
U=
\begin{pmatrix}
    0 & 1\\
    1 & 0
\end{pmatrix}
\]
so the isometry $T(\cA_X)\cong T(\cA_{X'})$ extends to an isometry
\[
\mukaiH(\cA_{X},\bZ)\cong\mukaiH(\cA_{X'},\bZ).
\]
Since $A(X')\subseteq H^4(X',\bZ)$ is saturated, the sublattice
$N(\cA_{X'})\subseteq\mukaiH(\cA_{X'},\bZ)$
is also saturated by \cite{AT14}*{Proposition~2.5~(2)}.
Therefore, 
\[
N(\cA_{X'})=N(\cA_{X'})^{\perp\perp}=T(\cA_{X'})^\perp\subseteq\mukaiH(\cA_{X'},\bZ)
\]
also contains a copy of the hyperbolic plane.
Again by \cite{AT14}*{Theorem~3.1}, this implies that $X'\in\cC_{d'}$ for some admissible $d'$.
\end{proof}

\bigskip
\bibliography{CremonaCubic_bib}
\bibliographystyle{alpha}

\ContactInfo
\end{document}